\documentclass[11pt, reqno]{amsart} \textwidth=14.9cm \oddsidemargin=1cm \evensidemargin=1cm
\usepackage{amsmath,amssymb,amscd,mathrsfs,stmaryrd,wasysym}
\usepackage{amsmath}
\usepackage{amsxtra}
\usepackage{amscd}
\usepackage{amsthm}
\usepackage{amsfonts}
\usepackage{amssymb}
\usepackage{eucal}
\usepackage{color}
\usepackage{graphicx}%
\newtheorem{theorem}{Theorem}[section]
\newtheorem{lemma}[theorem]{Lemma}
\newtheorem{proposition}[theorem]{Proposition}
\newtheorem{corollary}[theorem]{Corollary}
\theoremstyle{definition}
\newtheorem{definition}[theorem]{Definition}

\definecolor{A}{rgb}{.75,1,.75}

\numberwithin{equation}{section}

\begin{document}

\title[Iwahori-Matsumoto presentation]{An Iwahori-Matsumoto presentation of affine Yokonuma-Hecke algebras}
\author[Weideng Cui]{Weideng Cui}
\maketitle{\begin{center}
{\it \text{Dedicated to Professor George Lusztig on his seventieth birthday}}
\end{center}}

\begin{abstract}
We first present an Iwahori-Matsumoto presentation of affine Yokonuma-Hecke algebras $\widehat{Y}_{r,n}(q)$ to give a new proof of the fact, which was previously proved by Chlouveraki and S\'{e}cherre, that $\widehat{Y}_{r,n}(q)$ is a particular case of the pro-$p$-Iwahori-Hecke algebras defined by Vign\'eras; meanwhile, we give one application. Using the new presentation, we then give a third presentation of $\widehat{Y}_{r,n}(q),$ from which we immediately get an unexpected result, that is, the extended affine Hecke algebra of type $A$ is a subalgebra of the affine Yokonuma-Hecke algebra.
\end{abstract}



\maketitle
\medskip
\section{Introduction}
\subsection{}
Affine Hecke algebras $\widehat{\mathcal{H}}_{n}(q)$ were introduced by Iwahori and Matsumoto in [IM], in which they constructed an isomorphism between the convolution algebra $\mathbb{C}[I\backslash G(\mathbb{Q}_{p})/I]$ of $I$-bi-invariant compactly supported functions on $G(\mathbb{Q}_{p})$ and the specialization algebra $\widehat{\mathcal{H}}_{n}(q)|_{q=p},$ where $I$ is an Iwahori subgroup of $G(\mathbb{Q}_{p}).$ Later on, Bernstein found a totally different presentation of $\widehat{\mathcal{H}}_{n}(q)$ in terms of an alternative set of generators and relations, which is a $q$-analogue of the presentation of the extended affine Weyl group, as a semi-direct product of the finite Weyl group and a lattice of translations. The representation theory of affine Hecke algebras is very important, and has been studied extensively over the past few decades; see [KL, CG, Xi] and so on.

\subsection{}
When considering the convolution algebra $\mathbb{C}[I(1)\backslash GL_{n}(F)/I(1)]$ of compactly supported functions on $GL_{n}(F)$, where $F$ is a local non-Archimedean field and $I(1)$ is the pro-$p$-radical of an Iwahori subgroup, Vign\'eras [Vi1] introduced the pro-$p$-Iwahori-Hecke algebras. In a recent series of papers [Vi2-4], Vign\'eras defined and studied the pro-$p$-Iwahori-Hecke algebras associated to $p$-adic reductive groups of arbitrary type. In particular, she gave the Iwahori-Matsumoto presentation and Bernstein presentation of them, described their centers, and classified their supersingular simple modules among other things. Pro-$p$-Iwahori-Hecke algebras play an important role in the study of mod-$p$ representations of $p$-adic reductive groups.


\subsection{}
Yokonuma-Hecke algebras were introduced by Yokonuma \cite{Yo} as a centralizer algebra associated to the permutation representation of a finite Chevalley group $G$ with respect to a maximal unipotent subgroup of $G$. In order to study the representations of Yokonuma-Hecke algebras, Chlouveraki and Poulain d'Andecy [ChPA] defined and studied affine Yokonuma-Hecke algebras $\widehat{Y}_{r,n}(q)$. When $q^{2}$ is a power of a prime number $p$ and $r=q^{2}-1,$ one can verify that $\widehat{Y}_{r,n}(q)$ is isomorphic to the specialized pro-$p$-Iwahori-Hecke algebra associated to $GL_{n}(F).$ Later on, Chlouveraki and S\'{e}cherre [ChS] proved that the affine Yokonuma-Hecke algebra is a particular case of the pro-$p$-Iwahori-Hecke algebras by using their Bernstein presentations.

In [CW], we gave the classification of the simple $\widehat{Y}_{r,n}(q)$-modules as well as the classification of the simple modules of the cyclotomic Yokonuma-Hecke algebras over an algebraically closed field $\mathbb{K}$ of characteristic $p$ such that $p$ does not divide $r.$ Moreover, We [C1] and also Poulain d'Andecy [PA] proved that the affine Yokonuma-Hecke algebra is in fact isomorphic to a direct sum of matrix algebras over tensor products of extended affine Hecke algebras of type $A.$ Recently, we [C2] established an explicit categorical equivalence between affine Yokonuma-Hecke algebras and quiver Hecke algebras associated to disjoint copies of quivers of (affine) type $A,$ generalizing Rouquier's categorical equivalence theorem.

\subsection{}
In order to give a Frobenius type formula for the characters of Ariki-Koike algebras, Shoji [S] first defined a variation of the Ariki-Koike algebra, called the modified Ariki-Koike algebra in [SS], as a way of approximating the usual Ariki-Koike algebra. In general the two algebras are not isomorphic, but they are isomorphic if a certain separation condition holds. Later on, Espinoza and Ryom-Hansen [ER] proved that the Yokonuma-Hecke algebra is isomorphic to the modified Ariki-Koike algebra. Thus, they gave a new proof of the isomorphism theorem for Yokonuma-Hecke algebras, previously proved by Lusztig [Lu2] and also by Jacon-Poulain d'Andecy [JaPA], by using the isomorphism theorem for modified Ariki-Koike algebras established by Sawada-Shoji [SS] and independently by Hu-Stoll [HS].

\subsection{}
In this paper, in Section 2, we first present an Iwahori-Matsumoto presentation of affine Yokonuma-Hecke algebras $\widehat{Y}_{r,n}(q)$ to give a new proof of the fact, which was previously proved by Chlouveraki and S\'{e}cherre, that $\widehat{Y}_{r,n}(q)$ is a particular case of the pro-$p$-Iwahori-Hecke algebras defined by Vign\'eras; meanwhile, we develop one application, that is, we follow Jacon-Poulain d'Andecy's approach in [JaPA] to give a new proof of the isomorphism theorem for affine Yokonuma-Hecke algebras. In Section 3, we then introduce a third presentation of the affine Yokonuma-Hecke algebra using its Iwahori-Matsumoto presentation, from which we immediately obtain an unexpected result, that is, the extended affine Hecke algebra of type $A$ is a subalgebra of the affine Yokonuma-Hecke algebra.

This is the first paper of a series. In two papers [C3] and [C4] which will come very soon, we shall define and study two types of affine Yokonuma-Schur algebras associated to the two presentations of affine Yokonuma-Hecke algebras given in this paper; for example, we shall define the standard bases and canonical bases of them and classify the simple modules of them among other things.

\section{An Iwahori-Matsumoto presentation of $\widehat{Y}_{r,n}(q)$}

\subsection{Affine Yokonuma-Hecke algebras}
Let $r, n\in \mathbb{N},$ $r, n\geq1.$ Let $q$ be an indeterminate and let $\mathcal{R}=\mathbb{Z}[\frac{1}{r}][q,q^{-1}].$

\begin{definition}
The affine Yokonuma-Hecke algebra, denoted by $\widehat{Y}_{r,n}=\widehat{Y}_{r,n}(q)$, is an $\mathcal{R}$-associative algebra generated by the elements $\bar{t}_{1},\ldots,\bar{t}_{n},g_{1},\ldots,g_{n-1},X_{1}^{\pm1},$ in which the generators $\bar{t}_{1},\ldots,\bar{t}_{n},g_{1},$ $\ldots,g_{n-1}$ satisfy the following relations:
\begin{align}
g_ig_j&=g_jg_i \qquad \qquad\qquad\quad\hspace{0.3cm}\mbox{for all $i,j=1,\ldots,n-1$ such that $\vert i-j\vert \geq 2$;}\label{rel-def-Y1}\\[0.1em]
g_ig_{i+1}g_i&=g_{i+1}g_ig_{i+1} \qquad \quad\qquad\hspace{0.05cm}\mbox{for all $i=1,\ldots,n-2$;}\label{rel-def-Y2}\\[0.1em]
\bar{t}_i\bar{t}_j&=\bar{t}_j\bar{t}_i \qquad\qquad\qquad\qquad  \mbox{for all $i,j=1,\ldots,n$;}\label{rel-def-Y3}\\[0.1em]
g_i\bar{t}_j&=\bar{t}_{s_i(j)}g_i \quad \quad\qquad\qquad\hspace{0.25cm}\mbox{for all $i=1,\ldots,n-1$ and $j=1,\ldots,n$;}\label{rel-def-Y4}\\[0.1em]
\bar{t}_i^r&=1 \quad \qquad\qquad\qquad\qquad\mbox{for all $i=1,\ldots,n$;}\label{rel-def-Y5}\\[0.2em]
g_{i}^{2}&=1+(q-q^{-1})\bar{e}_{i}g_{i} \quad \hspace{0.48cm}\mbox{for all $i=1,\ldots,n-1$;}\label{rel-def-Y6}
\end{align}
where $s_{i}$ is the transposition $(i,i+1)$, and for each $1\leq i\leq n-1$,
$$\bar{e}_{i} :=\frac{1}{r}\sum\limits_{s=0}^{r-1}\bar{t}_{i}^{s}\bar{t}_{i+1}^{-s},$$
together with the following relations concerning the generators $X_{1}^{\pm1}$:
\begin{align}
X_{1}X_{1}^{-1}&=X_{1}^{-1}X_{1}=1;\label{rel-def-Y7}\\[0.1em]
g_{1}X_{1}g_{1}X_{1}&=X_{1}g_{1}X_{1}g_{1};\label{rel-def-Y8}\\[0.1em]
g_{i}X_{1}&=X_{1}g_{i} \qquad \qquad \quad\mbox{for all $i=2,\ldots,n-1$;}\label{rel-def-Y9}\\[0.1em]
\bar{t}_{j}X_{1}&=X_{1}\bar{t}_{j} \qquad \qquad\quad\mbox{for all $j=1,\ldots,n$.}\label{rel-def-Y10}
\end{align}
\end{definition}

By definition, we see that the elements $\bar{e}_{i}$'s are idempotents in $\widehat{Y}_{r,n},$ and the elements $g_{i}$'s are invertible with the inverse given by
\begin{equation}\label{inverse}
g_{i}^{-1}=g_{i}-(q-q^{-1})\bar{e}_{i}\quad\mbox{for~all}~i=1,\ldots,n-1.
\end{equation}

For each $w\in \mathfrak{S}_{n},$ let $w=s_{i_1}\cdots s_{i_{r}}$ be a reduced expression of $w.$ By Matsumoto's lemma, the element $g_{w} :=g_{i_1}g_{i_2}\cdots g_{i_{r}}$ does not depend on the choice of the reduced expression of $w$.

Let $i, k\in \{1,2,\ldots,n\}$ and set
\begin{equation}
\bar{e}_{i,k} :=\frac{1}{r}\sum\limits_{s=0}^{r-1}\bar{t}_{i}^{s}\bar{t}_{k}^{-s}.
\end{equation}
Note that $\bar{e}_{i,i}=1,$ $\bar{e}_{i,k}=\bar{e}_{k,i},$ and that $\bar{e}_{i,i+1}=\bar{e}_{i}.$ It can be easily checked that the following holds:
\begin{equation}\label{egge}
g_{i}\bar{e}_{j,k}=\bar{e}_{s_{i}(j),s_{i}(k)}g_{i}\quad\mbox{for $i=1,\ldots,n-1$ and $j,k=1,\ldots,n$}.
\end{equation}
In particular, we have $g_{i}\bar{e}_{i}=\bar{e}_{i}g_{i}$ for all $i=1,\ldots,n-1.$

We define the elements $X_{2},\ldots,X_{n}$ in $\widehat{Y}_{r,n}$ by induction:
\begin{equation}
X_{i+1} :=g_{i}X_{i}g_{i}\quad\mathrm{for}~i=1,\ldots,n-1.\label{X2n}
\end{equation}
Then it is proved in [ChPA, Lemma 1] that we have, for any $1\leq i\leq n-1$,
\begin{equation}
g_{i}X_{j}=X_{j}g_{i}\quad\mathrm{for}~j=1,2,\ldots,n~\mathrm{such~that}~j\neq i, i+1.\label{giXj}
\end{equation}
Moreover, by [ChPA, Proposition 1], we have that the elements $\bar{t}_{1},\ldots, \bar{t}_{n}, X_{1},\ldots, X_{n}$ form a commutative family, that is,
\begin{equation}
xy=yx\quad\mathrm{for~any}~x,y\in \{\bar{t}_{1},\ldots, \bar{t}_{n}, X_{1},\ldots, X_{n}\}.\label{xyyx}
\end{equation}

\subsection{An Iwahori-Matsumoto presentation}
Let $\widehat{W}$ be the extended affine Weyl group of type $A$, which is generated by $\rho,$ $s_{i}, 0\leq i\leq n-1$ satisfying the relations:
\begin{align}
\label{ECoxeter1}
s_{i}^{2}&=1,
\qquad\qquad\qquad\hspace{1.5mm}
s_{i}s_{j}=s_{j}s_{i}\hspace{3mm} \text{if $i-j\not\equiv \pm 1$ (mod $n$)};\\
\label{ECoxeter2}
\rho s_{\overline{i}}&=s_{\overline{i-1}}\rho,
\qquad\quad
s_{\overline{i}}s_{\overline{i+1}}s_{\overline{i}}=s_{\overline{i+1}}s_{\overline{i}}s_{\overline{i+1}}\hspace{3mm} \text{for $0\leq i\leq n-1$},
\end{align}
where $\overline{i}\in \{0,1,\ldots, n-1\}$ with $\overline{i}\equiv i\text{ (mod }n).$

Set $\mathfrak{X} :=\mathbb{Z}^{n},$ which is identified with a free abelian group generated by $X_{1},\ldots,X_{n}$ such that each element of $\mathfrak{X}$ can be written in the form $X^{\lambda} :=X_{1}^{\lambda_{1}}\cdots X_{n}^{\lambda_{n}}$ for $\lambda=(\lambda_{1},\ldots,\lambda_{n})\in \mathbb{Z}^{n}.$ The following lemma is well-known.
\begin{lemma}\label{coxeter-lemma-iso1}
We have an isomorphism of groups $\widehat{W}\cong \mathfrak{X}\rtimes \mathfrak{S}_{n},$ which is given by
\begin{align}
s_{i}&\mapsto s_{i}\quad\text{ for }1\leq i\leq n-1,\notag\\
s_{0}&\mapsto s_{n-1}\cdots s_{2}s_{1}s_{2}\cdots s_{n-1}X_{1}X_{n}^{-1},\label{coxeter-lemma-iso2}\\
\rho &\mapsto s_{n-1}\cdots s_{1}X_{1}.\notag
\end{align}
Its inverse sends $X_{1}$ to $s_{1}\cdots s_{n-1}\rho$ and $s_{i}$ to $s_{i}$ for $1\leq i\leq n-1.$
\end{lemma}

Let $\mathcal{T}=(\mathbb{Z}/r\mathbb{Z})^{n},$ which is a commutative group generated by $\bar{t}_{1},\ldots,\bar{t}_{n}$ with relations:
\begin{align*}
\bar{t}_{i}\bar{t}_{j}&=\bar{t}_{j}\bar{t}_{i}\qquad\mathrm{for ~all}~i,j=1,2,\ldots,n,\\
\bar{t}_{i}^{r}&=1\qquad\quad\mathrm{for ~all}~i=1,2,\ldots,n.
\end{align*}
We can write each element of $\mathcal{T}$ as $\bar{t}^{\beta}=\bar{t}_{1}^{\beta_{1}}\cdots \bar{t}_{n}^{\beta_{n}}$ for $\beta=(\beta_{1},\ldots,\beta_{n})$ with each $0\leq \beta_{i}\leq r-1.$ We consider the semi-direct product $\widetilde{W}_{r,n} :=(\mathcal{T}\times \mathfrak{X})\rtimes \mathfrak{S}_{n},$ in which every element can be written as $\bar{t}^{\beta}X^{\lambda}\sigma.$ Note that $\bar{t}^{\beta}$ and $X^{\lambda}$ commute with each other.

We also define a group $\widehat{W}_{r,n}$, which is generated by $t_{j},$ $1\leq j\leq n,$ $\rho,$ $s_{i}, 0\leq i\leq n-1$ satisfying the following relations:
\begin{align}
\label{ECoxeter1-1}
s_{i}^{2}&=1,
\qquad\qquad\qquad\hspace{1.5mm}
s_{i}s_{j}=s_{j}s_{i}\hspace{3mm} \text{if $i-j\not\equiv \pm 1$ (mod $n$)};\\
\label{ECoxeter2-2}
\rho s_{\overline{i}}&=s_{\overline{i-1}}\rho,
\qquad\quad
s_{\overline{i}}s_{\overline{i+1}}s_{\overline{i}}=s_{\overline{i+1}}s_{\overline{i}}s_{\overline{i+1}}\hspace{3mm} \text{for $0\leq i\leq n-1$};\\
\label{ECoxeter3-3}
t_{i}^{r}&=1,\qquad\quad \hspace{14.5mm} t_{i}t_{j}=t_{j}t_{i} \hspace{3mm} \text{ for }1\leq i,j\leq n;\\
\label{ECoxeter4-4}
s_{i}t_{j}&=t_{s_{i}(j)}s_{i}\qquad \quad\hspace{8mm}\rho t_{j}=t_{j-1}\rho \hspace{3mm} \text{for $1\leq i\leq n-1$},
\end{align}
where we set $t_{0} :=t_{n}.$

By generalizing Lemma \ref{coxeter-lemma-iso1}, we can easily get the following result.
\begin{lemma}\label{coxeter-lemma-iso2}
We have an isomorphism of groups $\widehat{W}_{r,n}\cong \widetilde{W}_{r,n}.$
\end{lemma}

\begin{definition}\label{Definition 2-4}
We define an $\mathcal{R}$-associative algebra $\widehat{H}_{r,n}^{\mathrm{aff}}$ generated by the elements $t_{1},\ldots,t_{n},$ $T_{s_{0}},\ldots,T_{s_{n-1}},$ $T_{\rho}^{\pm 1}$ with the following relations:
\begin{align}
t_i^r&=1 \quad \qquad\qquad\qquad\qquad\mbox{for all $1\leq i\leq n$;}\label{rel-def-Heckealg1}\\[0.1em]
t_it_j&=t_jt_i \qquad\qquad\qquad\qquad\hspace{0.4mm}  \mbox{for all $1\leq i, j\leq n$;}\label{rel-def-Heckealg2}\\[0.1em]
T_{s_{i}}t_j&=t_{s_i(j)}T_{s_{i}} \quad \quad\qquad\qquad\hspace{1mm}\mbox{for all $1\leq i\leq n-1$ and $1\leq j\leq n$;}\label{rel-def-Heckealg3}\\[0.1em]
T_{\rho}t_j&=t_{j-1}T_{\rho} \quad \quad\qquad\qquad\hspace{3.1mm}\mbox{for all $1\leq j\leq n$;}\label{rel-def-Heckealg4}\\[0.1em]
T_{\rho}T_{s_{\overline{i}}}&=T_{s_{\overline{i-1}}}T_{\rho}  \quad \quad\qquad\qquad\hspace{1.8mm}\mbox{for all $0\leq i\leq n-1 $;}\label{rel-def-Heckealg5}\\[0.1em]
T_{s_{i}}T_{s_{j}}&=T_{s_{j}}T_{s_{i}}\quad \quad\qquad\qquad\hspace{4.15mm}  \text{if $i-j\not\equiv \pm 1$ (mod $n$)};\label{rel-def-Heckealg6}\\[0.1em]
T_{s_{\overline{i}}}T_{s_{\overline{i+1}}}T_{s_{\overline{i}}}&=T_{s_{\overline{i+1}}}T_{s_{\overline{i}}}T_{s_{\overline{i+1}}}\quad \quad\qquad \hspace{0.8mm}  \text{if $0\leq i\leq n-1$ and $n\geq 3$};\label{rel-def-Heckealg7}\\[0.1em]
T_{s_{i}}^{2}&=1+(q-q^{-1})e_{i}T_{s_{i}} \quad \hspace{3.8mm}\mbox{for all $0\leq i\leq n-1$;}\label{rel-def-Heckealg8}\\[0.1em]
T_{\rho}T_{\rho}^{-1}&=T_{\rho}^{-1}T_{\rho}=1,\label{rel-def-Heckealg4-com}
\end{align}
where $t_{0} :=t_{n}$ and for each $0\leq i\leq n-1$,
$$e_{i} :=\frac{1}{r}\sum\limits_{s=0}^{r-1}t_{i}^{s}t_{i+1}^{-s}.$$
\end{definition}

We now state the main result of this section, which can be regarded as a generalization of the isomorphism theorem between the Iwahori-Matsumoto presentation and Bernstein presentation of an extended affine Hecke algebra of type $A$.
\begin{theorem}\label{iwahori-matsumoto-pre}
We have an $\mathcal{R}$-algebra isomorphism $\Phi :\widehat{H}_{r,n}^{\mathrm{aff}}\rightarrow \widehat{Y}_{r,n}$ given by
\[
  \Phi: \quad   \begin{array}{ccc}
t_{j} &\longmapsto &\hspace{-6mm}\bar{t}_{j}\qquad\qquad\hspace{5mm}\text{ for }1\leq j\leq n, \\
T_{s_{i}} &\longmapsto &g_{i}\quad\qquad\hspace{9mm}\text{ for }1\leq i\leq n-1,\\
T_{s_{0}} &\longmapsto &X_{1}^{-1}X_{n}(g_{n-1}\cdots g_{2}g_{1}g_{2}\cdots g_{n-1})^{-1},\\
T_{\rho}  &\longmapsto &\hspace{-3.5cm}g_{n-1}\cdots g_{1}X_{1},\\
T_{\rho}^{-1}  &\longmapsto &\hspace{-3.5cm}X_{1}^{-1}g_{1}^{-1}\cdots g_{n-1}^{-1}
\end{array}
\]
with the inverse $\Psi :\widehat{Y}_{r,n}\rightarrow \widehat{H}_{r,n}^{\mathrm{aff}}$ defined by
\[
  \Psi: \quad   \begin{array}{ccc}
\bar{t}_{j} &\longmapsto &\hspace{-6mm}t_{j}\qquad\qquad\hspace{5mm}\text{ for }1\leq j\leq n, \\
g_{i} &\longmapsto &\hspace{-1mm}T_{s_{i}}\quad\qquad\hspace{7.5mm}\text{ for }1\leq i\leq n-1,\\
X_{1}  &\longmapsto &\hspace{-3cm}T_{s_{1}}^{-1}\cdots T_{s_{n-1}}^{-1}T_{\rho},\\
X_{1}^{-1}  &\longmapsto &\hspace{-3cm}T_{\rho}^{-1}T_{s_{n-1}}\cdots T_{s_{1}}.
\end{array}
\]
\end{theorem}
\begin{proof}
We extend $\Phi$ and $\Psi$ defined on the generators to algebra homomorphisms. We need to show that $\Phi$ and $\Psi$ preserve the defining relations of $\widehat{H}_{r,n}^{\mathrm{aff}}$ and $\widehat{Y}_{r,n}$, respectively.

By \eqref{rel-def-Y3}-\eqref{rel-def-Y5}, it is obvious that $\Phi$ preserves \eqref{rel-def-Heckealg1}-\eqref{rel-def-Heckealg3}.

To show that $\Phi$ preserves \eqref{rel-def-Heckealg4}, it suffices to prove that $g_{n-1}\cdots g_{1}X_{1}\bar{t}_{j}=\bar{t}_{j-1}g_{n-1}\cdots g_{1}X_{1}$ for $1\leq j\leq n,$ where we set $\bar{t}_{0} :=\bar{t}_{n}.$ It easily follows from \eqref{rel-def-Y4} and \eqref{xyyx}.

To show that $\Phi$ preserves \eqref{rel-def-Heckealg5}, it suffices to prove that
\begin{align}\label{main-theorem-iden1}
g_{n-1}\cdots g_{1}X_{1}\cdot \Phi(T_{s_{\overline{i}}})=\Phi(T_{s_{\overline{i-1}}})\cdot g_{n-1}\cdots g_{1}X_{1}\quad \text{ for }0\leq i\leq n-1.
\end{align}
Note that
\begin{align}\label{main-theorem-iden2}
\Phi(T_{s_{0}})&=X_{1}^{-1}X_{n}g_{n-1}^{-1}\cdots g_{2}^{-1}g_{1}^{-1}g_{2}^{-1}\cdots g_{n-1}^{-1}\notag\\
&=X_{1}^{-1}g_{n-1}\cdots g_{2}g_{1}X_{1}g_{2}^{-1}\cdots g_{n-1}^{-1}\notag\\
&=X_{1}^{-1}g_{n-1}\cdots g_{2}g_{1}g_{2}^{-1}\cdots g_{n-1}^{-1}X_{1}.
\end{align}

For $i=0,$ in order to show \eqref{main-theorem-iden1}, it suffices to prove
\begin{align}\label{main-theorem-iden3}
g_{n-1}\cdots g_{2}g_{1}g_{2}^{-1}\cdots g_{n-1}^{-1}=g_{1}^{-1}g_{2}^{-1}\cdots g_{n-2}^{-1}g_{n-1}\cdots g_{2}g_{1}.
\end{align}
By $g_{i+1}g_{i}g_{i+1}^{-1}=g_{i}^{-1}g_{i+1}g_{i}$ for $1\leq i\leq n-2,$ we have
\begin{align*}
g_{n-1}\cdots g_{2}g_{1}g_{2}^{-1}\cdots g_{n-1}^{-1}&=g_{n-1}\cdots g_{3}g_{1}^{-1}g_{2}g_{1}g_{3}^{-1}\cdots g_{n-1}^{-1}\notag\\
&=g_{1}^{-1}g_{n-1}\cdots g_{4}g_{2}^{-1}g_{3}g_{2}g_{4}^{-1}\cdots g_{n-1}^{-1}g_{1}\notag\\
&=\cdots\cdots\notag\\
&=g_{1}^{-1}g_{2}^{-1}\cdots g_{n-2}^{-1}g_{n-1}\cdots g_{2}g_{1}.
\end{align*}

For $i=1,$ in order to show \eqref{main-theorem-iden1}, it suffices to prove
\begin{align}\label{main-theorem-iden4}
g_{n-1}\cdots g_{2}g_{1}X_{1}g_{1}X_{1}^{-1}g_{1}^{-1}\cdots g_{n-1}^{-1}=X_{1}^{-1}g_{n-1}\cdots g_{2}g_{1}X_{1}g_{2}^{-1}\cdots g_{n-1}^{-1},
\end{align}
which follows from \eqref{giXj} and the commutativity of $g_{1}X_{1}g_{1}$ and $X_{1}^{-1}$.

For $2\leq i\leq n-1,$ in order to show \eqref{main-theorem-iden1}, it suffices to prove
\begin{align}\label{main-theorem-iden5}
g_{n-1}\cdots g_{2}g_{1}X_{1}g_{i}=g_{i-1}g_{n-1}\cdots g_{2}g_{1}X_{1}.
\end{align}
Note that
\begin{align*}
g_{i-1}^{-1}g_{n-1}\cdots g_{2}g_{1}X_{1}g_{i}&=g_{n-1}\cdots g_{i+1}g_{i-1}^{-1}g_{i}g_{i-1}g_{i-2}\cdots g_{1}X_{1}g_{i}\notag\\
&=g_{n-1}\cdots g_{i+1}g_{i}g_{i-1}g_{i}^{-1}g_{i-2}\cdots g_{1}X_{1}g_{i}\notag\\
&=g_{n-1}\cdots g_{i+1}g_{i}g_{i-1}g_{i-2}\cdots g_{1}g_{i}^{-1}X_{1}g_{i}\notag\\
&=g_{n-1}\cdots g_{2}g_{1}X_{1}.
\end{align*}
We see that \eqref{main-theorem-iden5} holds.

To show that $\Phi$ preserves \eqref{rel-def-Heckealg6}, it suffices to prove that $\Phi(T_{s_{0}})\Phi(T_{s_{j}})=\Phi(T_{s_{j}})\Phi(T_{s_{0}})$ for $2\leq j\leq n-2,$ which is equivalent to the following identity:
\begin{align}\label{main-theorem-iden6}
g_{j}\cdot X_{1}^{-1}g_{n-1}\cdots g_{2}g_{1}g_{2}^{-1}\cdots g_{n-1}^{-1}X_{1}=X_{1}^{-1}g_{n-1}\cdots g_{2}g_{1}g_{2}^{-1}\cdots g_{n-1}^{-1}X_{1}\cdot g_{j}.
\end{align}
Note that
\begin{align*}
&g_{j}\cdot X_{1}^{-1}g_{n-1}\cdots g_{2}g_{1}g_{2}^{-1}\cdots g_{n-1}^{-1}X_{1}\\
=&X_{1}^{-1}g_{n-1}\cdots g_{j+2}g_{j}g_{j+1}g_{j}g_{j-1}\cdots g_{2}g_{1}g_{2}^{-1}\cdots g_{n-1}^{-1}X_{1}\\
=&X_{1}^{-1}g_{n-1}\cdots g_{j+2}g_{j+1}g_{j}g_{j+1}g_{j-1}\cdots g_{2}g_{1}g_{2}^{-1}\cdots g_{n-1}^{-1}X_{1}\\
=&X_{1}^{-1}g_{n-1}\cdots g_{j+2}g_{j+1}g_{j}g_{j-1}\cdots g_{2}g_{1}g_{2}^{-1}\cdots g_{j-1}^{-1}g_{j+1}g_{j}^{-1}g_{j+1}^{-1}g_{j+2}^{-1}\cdots g_{n-1}^{-1}X_{1}\\
=&X_{1}^{-1}g_{n-1}\cdots g_{2}g_{1}g_{2}^{-1}\cdots g_{j-1}^{-1}g_{j}^{-1}g_{j+1}^{-1}g_{j}g_{j+2}^{-1}\cdots g_{n-1}^{-1}X_{1}\\
=&X_{1}^{-1}g_{n-1}\cdots g_{2}g_{1}g_{2}^{-1}\cdots g_{n-1}^{-1}X_{1}\cdot g_{j}.
\end{align*}
We see that \eqref{main-theorem-iden6} holds.

To show that $\Phi$ preserves \eqref{rel-def-Heckealg7}, it suffices to prove that
\begin{align}\label{main-theorem-iden7}
\Phi(T_{s_{1}})\Phi(T_{s_{0}})\Phi(T_{s_{1}})=\Phi(T_{s_{0}})\Phi(T_{s_{1}})\Phi(T_{s_{0}})
\end{align}
and
\begin{align}\label{main-theorem-iden8}
\Phi(T_{s_{n-1}})\Phi(T_{s_{0}})\Phi(T_{s_{n-1}})=\Phi(T_{s_{0}})\Phi(T_{s_{n-1}})\Phi(T_{s_{0}}).
\end{align}

We first show that \eqref{main-theorem-iden7} holds. We have
\begin{align}\label{main-theorem-iden9}
\Phi(T_{s_{1}})\Phi(T_{s_{0}})\Phi(T_{s_{1}})&=g_{1}\cdot X_{1}^{-1}X_{n}g_{n-1}^{-1}\cdots g_{2}^{-1}g_{1}^{-1}g_{2}^{-1}\cdots g_{n-1}^{-1}\cdot g_{1}\notag\\
&=g_{1}\cdot X_{1}^{-1}X_{n}g_{n-1}^{-1}\cdots g_{3}^{-1}g_{1}^{-1}g_{2}^{-1}g_{1}^{-1}g_{3}^{-1}\cdots g_{n-1}^{-1}\cdot g_{1}\notag\\
&=g_{1}X_{1}^{-1}X_{n}g_{1}^{-1}\cdot g_{n-1}^{-1}\cdots g_{3}^{-1}g_{2}^{-1}g_{3}^{-1}\cdots g_{n-1}^{-1},
\end{align}
and
\begin{align}\label{main-theorem-iden10}
&\Phi(T_{s_{0}})\Phi(T_{s_{1}})\Phi(T_{s_{0}})\notag\\
=&X_{1}^{-1}X_{n}g_{n-1}^{-1}\cdots g_{2}^{-1}g_{1}^{-1}g_{2}^{-1}\cdots g_{n-1}^{-1}\cdot g_{1}\cdot X_{1}^{-1}X_{n}g_{n-1}^{-1}\cdots g_{2}^{-1}g_{1}^{-1}g_{2}^{-1}\cdots g_{n-1}^{-1}\notag\\
=&X_{1}^{-1}X_{n}g_{n-1}^{-1}\cdots g_{3}^{-1}g_{1}^{-1}g_{2}^{-1}g_{1}^{-1}g_{3}^{-1}\cdots g_{n-1}^{-1}\cdot g_{1}\cdot X_{1}^{-1}g_{n-1}\cdots g_{2}g_{1}X_{1}g_{2}^{-1}\cdots g_{n-1}^{-1}\notag\\
=&X_{1}^{-1}X_{n}g_{n-1}^{-1}\cdots g_{3}^{-1}g_{1}^{-1}g_{2}^{-1}g_{3}^{-1}\cdots g_{n-1}^{-1}\cdot X_{1}^{-1}g_{n-1}\cdots g_{2}g_{1}X_{1}g_{2}^{-1}\cdots g_{n-1}^{-1}\notag\\
=&X_{1}^{-1}X_{n}g_{n-1}^{-1}\cdots g_{3}^{-1}g_{1}^{-1}X_{1}^{-1}g_{1}X_{1}g_{2}^{-1}\cdots g_{n-1}^{-1}\notag\\
=&X_{1}^{-1}X_{n}g_{1}^{-1}X_{1}^{-1}g_{1}X_{1}\cdot g_{n-1}^{-1}\cdots g_{3}^{-1}g_{2}^{-1}\cdots g_{n-1}^{-1}
\end{align}
By \eqref{main-theorem-iden9} and \eqref{main-theorem-iden10}, in order to show \eqref{main-theorem-iden7}, it suffices to prove that \[g_{1}X_{1}^{-1}g_{1}^{-1}=X_{1}^{-1}g_{1}^{-1}X_{1}^{-1}g_{1}X_{1},\]
which follows from
\[g_{1}^{-1}X_{1}^{-1}g_{1}^{-1}X_{1}^{-1}g_{1}X_{1}g_{1}=(g_{1}X_{1}g_{1})^{-1}\cdot X_{1}^{-1}\cdot (g_{1}X_{1}g_{1})=X_{2}^{-1}X_{1}^{-1}X_{2}=X_{1}^{-1}.\]

Next we show that \eqref{main-theorem-iden8} holds. We have
\begin{align}\label{main-theorem-iden11}
\Phi(T_{s_{n-1}})\Phi(T_{s_{0}})\Phi(T_{s_{n-1}})=&g_{n-1}X_{1}^{-1}X_{n}g_{n-1}^{-1}\cdots g_{2}^{-1}g_{1}^{-1}g_{2}^{-1}\cdots g_{n-1}^{-1}g_{n-1}\notag\\
=&X_{1}^{-1}g_{n-1}X_{n}g_{n-1}^{-1}\cdots g_{2}^{-1}g_{1}^{-1}g_{2}^{-1}\cdots g_{n-2}^{-1},
\end{align}
and
\begin{align}\label{main-theorem-iden12}
&\Phi(T_{s_{0}})\Phi(T_{s_{n-1}})\Phi(T_{s_{0}})\notag\\
=&X_{1}^{-1}X_{n}g_{n-1}^{-1}\cdots g_{2}^{-1}g_{1}^{-1}g_{2}^{-1}\cdots g_{n-1}^{-1}\cdot g_{n-1}\cdot X_{1}^{-1}X_{n}g_{n-1}^{-1}\cdots g_{2}^{-1}g_{1}^{-1}g_{2}^{-1}\cdots g_{n-1}^{-1}\notag\\
=&X_{1}^{-1}g_{n-1}\cdots g_{2}g_{1}X_{1}g_{2}^{-1}\cdots g_{n-2}^{-1}\cdot X_{1}^{-1}X_{n}g_{n-1}^{-1}\cdots g_{2}^{-1}g_{1}^{-1}g_{2}^{-1}\cdots g_{n-1}^{-1}\notag\\
=&X_{1}^{-1}g_{n-1}X_{n}g_{n-2}\cdots g_{2}g_{1}g_{2}^{-1}\cdots g_{n-2}^{-1}\cdot g_{n-1}^{-1}\cdots g_{2}^{-1}g_{1}^{-1}g_{2}^{-1}\cdots g_{n-1}^{-1}.
\end{align}
By \eqref{main-theorem-iden11} and \eqref{main-theorem-iden12}, in order to show \eqref{main-theorem-iden8}, it suffices to prove that
\begin{align}\label{main-theorem-iden13}
g_{n-1}^{-1}\cdots g_{2}^{-1}g_{1}^{-1}g_{2}^{-1}\cdots g_{n-2}^{-1}=g_{n-2}\cdots g_{2}g_{1}g_{2}^{-1}\cdots g_{n-2}^{-1}\cdot g_{n-1}^{-1}\cdots g_{1}^{-1}\cdots g_{n-1}^{-1}.
\end{align}

By \eqref{rel-def-Y2}, we have
\[g_{n-2}^{-1}\cdots g_{2}^{-1}g_{1}^{-1}g_{2}^{-1}\cdots g_{n-2}^{-1}=g_{1}^{-1}g_{2}^{-1}\cdots g_{n-3}^{-1}g_{n-2}^{-1}g_{n-3}^{-1}\cdots g_{2}^{-1}g_{1}^{-1},\]
and
\[g_{n-1}^{-1}\cdots g_{2}^{-1}g_{1}^{-1}g_{2}^{-1}\cdots g_{n-1}^{-1}=g_{1}^{-1}g_{2}^{-1}\cdots g_{n-2}^{-1}g_{n-1}^{-1}g_{n-2}^{-1}\cdots g_{2}^{-1}g_{1}^{-1}.\]
Thus, in order to show \eqref{main-theorem-iden13}, it suffices to prove that
\begin{align}\label{main-theorem-iden14}
g_{n-1}^{-1}g_{1}^{-1}g_{2}^{-1}\cdots g_{n-3}^{-1}=g_{n-2}\cdots g_{2}g_{1}g_{2}^{-1}\cdots g_{n-2}^{-1}\cdot g_{1}^{-1}\cdots g_{n-2}^{-1}g_{n-1}^{-1},
\end{align}
that is,
\begin{align}\label{main-theorem-iden15}
g_{1}^{-1}g_{2}^{-1}\cdots g_{n-3}^{-1}=g_{n-2}\cdots g_{2}g_{1}g_{2}^{-1}\cdots g_{n-2}^{-1}\cdot g_{1}^{-1}\cdots g_{n-2}^{-1},
\end{align}
which follows from \eqref{main-theorem-iden3}.

By \eqref{main-theorem-iden2}, we have
\begin{align}\label{main-theorem-iden16}
\Phi(T_{s_{0}})^{2}&=X_{1}^{-1}g_{n-1}\cdots g_{2}g_{1}g_{2}^{-1}\cdots g_{n-1}^{-1}X_{1}\cdot X_{1}^{-1}g_{n-1}\cdots g_{2}g_{1}g_{2}^{-1}\cdots g_{n-1}^{-1}X_{1}\notag\\
=&X_{1}^{-1}g_{n-1}\cdots g_{2}g_{1}^{2}g_{2}^{-1}\cdots g_{n-1}^{-1}X_{1}\notag\\
=&X_{1}^{-1}g_{n-1}\cdots g_{2}(1+(q-q^{-1})\bar{e}_{1}g_{1})g_{2}^{-1}\cdots g_{n-1}^{-1}X_{1}\notag\\
=&1+(q-q^{-1})\bar{e}_{n, 1}X_{1}^{-1}g_{n-1}\cdots g_{2}g_{1}g_{2}^{-1}\cdots g_{n-1}^{-1}X_{1}\notag\\
=&1+(q-q^{-1})\bar{e}_{n, 1}\Phi(T_{s_{0}}),
\end{align}
which shows that $\Phi$ preserves \eqref{rel-def-Heckealg8}.

Next we show that $\Psi$ preserve the defining relations of $\widehat{Y}_{r,n}.$ It is obvious that $\Psi$ preserve the relations \eqref{rel-def-Y1}-\eqref{rel-def-Y7}.

In order to show that $\Psi$ preserve \eqref{rel-def-Y8}, it suffices to prove that
\begin{align}\label{main-theorem-iden17}
T_{s_{1}}^{-1}\cdots T_{s_{n-1}}^{-1}T_{\rho}\cdot T_{s_{1}} \cdot T_{s_{1}}^{-1}&\cdots T_{s_{n-1}}^{-1}T_{\rho}\cdot T_{s_{1}}\notag\\
&=T_{s_{1}}\cdot T_{s_{1}}^{-1}\cdots T_{s_{n-1}}^{-1}T_{\rho}\cdot T_{s_{1}}\cdot T_{s_{1}}^{-1}\cdots T_{s_{n-1}}^{-1}T_{\rho}.
\end{align}
By \eqref{rel-def-Heckealg5}, we have \[T_{\rho}\cdot T_{s_{2}}^{-1}\cdots T_{s_{n-1}}^{-1}T_{\rho}\cdot T_{s_{1}}=T_{s_{1}}^{-1}\cdots T_{s_{n-2}}^{-1}T_{\rho}^{2}T_{s_{1}}=T_{s_{1}}^{-1}\cdots T_{s_{n-2}}^{-1}T_{s_{n-1}}T_{\rho}^{2}.\]
Thus, in order to prove \eqref{main-theorem-iden17}, it suffices to show that
\begin{align}\label{main-theorem-iden18}
T_{s_{1}}^{-1}\cdots T_{s_{n-1}}^{-1}T_{s_{1}}^{-1}\cdots T_{s_{n-2}}^{-1}T_{s_{n-1}}=T_{s_{2}}^{-1}\cdots T_{s_{n-1}}^{-1}T_{s_{1}}^{-1}\cdots T_{s_{n-2}}^{-1},
\end{align}
that is,
\begin{align}\label{main-theorem-iden19}
T_{s_{1}}^{-1}\cdots T_{s_{n-2}}^{-1}T_{s_{n-1}}\cdots T_{s_{1}}=T_{s_{n-1}}\cdots T_{s_{1}}\cdot T_{s_{2}}^{-1}\cdots T_{s_{n-1}}^{-1},
\end{align}
which follows from an argument similar to \eqref{main-theorem-iden3}.

In order to show that $\Psi$ preserve \eqref{rel-def-Y9}, it suffices to prove that
\begin{align}\label{main-theorem-iden20}
T_{s_{i}}\cdot T_{s_{1}}^{-1}\cdots T_{s_{n-1}}^{-1}T_{\rho}=T_{s_{1}}^{-1}\cdots T_{s_{n-1}}^{-1}T_{\rho}\cdot T_{s_{i}}\quad \text{ for }2\leq i\leq n-1.
\end{align}
By \eqref{rel-def-Heckealg5}, we have $T_{\rho}\cdot T_{s_{i}}=T_{s_{i-1}}\cdot T_{\rho}$ for $2\leq i\leq n-1.$ Thus, in order to show \eqref{main-theorem-iden20}, it suffices to prove that
\begin{align}\label{main-theorem-iden21}
T_{s_{i}}\cdot T_{s_{1}}^{-1}\cdots T_{s_{n-1}}^{-1}=T_{s_{1}}^{-1}\cdots T_{s_{n-1}}^{-1}T_{s_{i-1}},
\end{align}
which follows from the identity $T_{s_{i}}T_{s_{i-1}}^{-1}T_{s_{i}}^{-1}=T_{s_{i-1}}^{-1}T_{s_{i}}^{-1}T_{s_{i-1}}.$

Finally we show that $\Psi$ preserve \eqref{rel-def-Y10}. It suffices to prove that
\begin{align}\label{main-theorem-iden22}
T_{s_{1}}^{-1}\cdots T_{s_{n-1}}^{-1}T_{\rho}\cdot t_{j}=t_{j}\cdot T_{s_{1}}^{-1}\cdots T_{s_{n-1}}^{-1}T_{\rho}\quad \text{ for }1\leq j\leq n.
\end{align}

For $j=1,$ we have $T_{s_{1}}^{-1}\cdots T_{s_{n-1}}^{-1}T_{\rho}\cdot t_{1}=T_{s_{1}}^{-1}\cdots T_{s_{n-1}}^{-1}t_{n}\cdot T_{\rho}=\cdots=t_{1}T_{s_{1}}^{-1}\cdots T_{s_{n-1}}^{-1}T_{\rho}.$

For $2\leq j\leq n,$ we have
\begin{align*}
T_{s_{1}}^{-1}\cdots T_{s_{n-1}}^{-1}T_{\rho}\cdot t_{j}&=T_{s_{1}}^{-1}\cdots T_{s_{n-1}}^{-1}t_{j-1}\cdot T_{\rho}\\
&=T_{s_{1}}^{-1}\cdots T_{s_{j-2}}^{-1}t_{j}T_{s_{j-1}}^{-1}T_{s_{j}}^{-1}\cdot T_{s_{n-1}}^{-1}\cdot T_{\rho}\\
&=t_{j}\cdot T_{s_{1}}^{-1}\cdots T_{s_{n-1}}^{-1}T_{\rho}.
\end{align*}
We see that \eqref{main-theorem-iden22} holds.

It is obvious that $\Phi\circ\Psi=$Id and $\Psi\circ\Phi=$Id. Thus, $\Phi$ and $\Psi$ establish an isomorphism of algebras.
\end{proof}

Let $S :=\{s_{1},\ldots,s_{n-1}\}$ and $S^{\mathrm{aff}} :=\{s_{0}, s_{1},\ldots,s_{n-1}\}$. Let $W^{\mathrm{aff}}$ be the subgroup of $\widehat{W}$ generated by $s_{0}, s_{1},\ldots, s_{n-1},$ which is exactly the affine Weyl group of type $A$. It is well-known that $W^{\mathrm{aff}}$ is a Coxeter group with a length function $\ell.$ We extend the length function $\ell$ from $W^{\mathrm{aff}}$ to $\widehat{W}$ by letting $\ell(\rho^{k}w)=\ell(w)$ for any $k\in \mathbb{Z}$ and $w\in W^{\mathrm{aff}}.$ We further extend the length function $\ell$ from $\widehat{W}$ to $\widehat{W}_{r,n}$ by setting $\ell(t\widehat{w})=\ell(\widehat{w})$ for any $t\in \mathcal{T}$ and $\widehat{w}\in \widehat{W},$ where $\mathcal{T}$ is identified with the subgroup of $\widehat{W}_{r,n}$ generated by the elements $t_{1},\ldots,t_{n}$ by Lemma \ref{coxeter-lemma-iso2}.

For each $\underline{w}\in \widehat{W}_{r,n},$ let $\underline{w}=t\rho^{k}s_{i_1}\cdots s_{i_{r}}$ be a reduced expression of $\underline{w}.$ From the presentation of $\widehat{H}_{r,n}^{\mathrm{aff}}$ given in Definition \ref{Definition 2-4}, we see that $T_{\underline{w}} :=T_{t}T_{\rho}^{k}T_{s_{i_1}}\cdots T_{s_{i_{r}}}$ is well-defined, that is, it does not depend on the choice of the reduced expression of $\underline{w}$, and here we set $T_{t} :=t\in \mathcal{T}.$

We first state the following lemma, which can be regarded as a generalization of [Ju, Lemmas 3 and 5].
\begin{lemma}\label{lemma-juyumaya-gene}
In $\widehat{W}_{r,n}$, if there exist $s_{i}, s_{j}$ $(0\leq i, j\leq n-1)$ and $\underline{w}\in \widehat{W}_{r,n}$ such that $\ell(s_{i}\underline{w}s_{j})=\ell(\underline{w})$ and $\ell(s_{i}\underline{w})=\ell(\underline{w}s_{j}),$ then we have

$(1)$ $e_{i}s_{i}\underline{w}=e_{i}\underline{w}s_{j}.$

$(2)$ $s_{i}\underline{w}e_{j}=e_{i}\underline{w}s_{j}.$

$(3)$ $e_{i}\underline{w}=\underline{w}e_{j}.$
\end{lemma}
\begin{proof}
From the definition, we can get that $\rho e_{\overline{i}}=e_{\overline{i-1}}\rho$, $s_{i}e_{i}=e_{i}s_{i}$ and ${}^{s_{i}}\!t e_{i}=e_{i}{}^{s_{i}}\!t=e_{i}t$ for any $0\leq i\leq n-1.$

$(1)$ We assume that the reduced expression of $\underline{w}$ is $\underline{w}=t\rho^{k}w$ for some $k\in \mathbb{Z}$ and $w\in W^{\mathrm{aff}}.$ By assumption, we have
\[\ell(s_{i}\underline{w}s_{j})=\ell(s_{i}t\rho^{k}ws_{j})=\ell({}^{s_{i}}\!t\rho^{k}s_{\overline{i+k}}ws_{j})
=\ell(s_{\overline{i+k}}ws_{j})=\ell(t\rho^{k}w)=\ell(w).\]
Similarly, we have
\[\ell(s_{i}\underline{w})=\ell(s_{i}t\rho^{k}w)=\ell({}^{s_{i}}\!t\rho^{k}s_{\overline{i+k}}w)=\ell(s_{\overline{i+k}}w)=\ell(t\rho^{k}ws_{j})
=\ell(ws_{j}).\]
By [Lu1, Proposition 1.10], we have $s_{\overline{i+k}}w=ws_{j}.$ Thus, we have
\[e_{i}s_{i}\underline{w}=e_{i}s_{i}t\rho^{k}w=e_{i}{}^{s_{i}}\!t\rho^{k}s_{\overline{i+k}}w=e_{i}t\rho^{k}ws_{j}=e_{i}\underline{w}s_{j}.\]

$(2)$ We prove it by induction on $\ell(\underline{w})$. If $\ell(\underline{w})=1$, then we assume that $\underline{w}=t\rho^{h}s_{k}$ for some $h\in \mathbb{Z}$ and $0\leq k\leq n-1.$

If $k=\overline{i+h}$, by the equality
\[\ell(s_{i}\underline{w})=\ell(s_{i}t\rho^{h}s_{k})=\ell({}^{s_{i}}\!t\rho^{h}s_{\overline{i+h}}s_{k})=0=\ell(t\rho^{h}s_{k}s_{j})=\ell(\underline{w}s_{j}),\]
we must have $k=j.$ Thus, we get that
\[s_{i}\underline{w}e_{j}=s_{i}t\rho^{h}s_{k}e_{j}={}^{s_{i}}\!t\rho^{h}s_{\overline{i+h}}s_{k}e_{j}={}^{s_{i}}\!t\rho^{h}e_{\overline{i+h}}
={}^{s_{i}}\!te_{i}\rho^{h}=e_{i}t\rho^{h}=e_{i}\underline{w}s_{j}.\]

If $k\neq\overline{i+h}$, by the equality
\[\ell(s_{i}\underline{w}s_{j})=\ell(s_{i}t\rho^{h}s_{k}s_{j})=\ell({}^{s_{i}}\!t\rho^{h}s_{\overline{i+h}}s_{k}s_{j})
=\ell(s_{\overline{i+h}}s_{k}s_{j})=\ell(s_{k})=\ell(\underline{w}),\]
we must have $j=\overline{i+h},$ and hence $k-\overline{i+h}\not\equiv \pm 1$ $(\text{mod }n).$ So we have $s_{\overline{i+h}}s_{k}=s_{k}s_{\overline{i+h}}$ and $s_{k}e_{\overline{i+h}}=e_{\overline{i+h}}s_{k}.$ Thus, we get
\begin{align*}
s_{i}\underline{w}e_{j}&=s_{i}t\rho^{h}s_{k}e_{j}={}^{s_{i}}\!t\rho^{h}s_{\overline{i+h}}s_{k}e_{\overline{i+h}}
={}^{s_{i}}\!t\rho^{h}e_{\overline{i+h}}s_{\overline{i+h}}s_{k}\\
&={}^{s_{i}}\!te_{i}\rho^{h}s_{k}s_{\overline{i+h}}=e_{i}t\rho^{h}s_{k}s_{j}=e_{i}\underline{w}s_{j}.
\end{align*}

Now we assume that the equality $(2)$ is true if $\ell(\underline{w})< n.$ We suppose that $\underline{w}=t\rho^{h}s_{k_{1}}\cdots s_{k_{n}}$ is a reduced expression of $\underline{w}.$

If $\ell(\underline{w})> \ell(s_{i}\underline{w}),$ since we have $s_{i}\underline{w}=s_{i}t\rho^{h}s_{k_{1}}\cdots s_{k_{n}}={}^{s_{i}}\!t\rho^{h}s_{\overline{i+h}}s_{k_{1}}\cdots s_{k_{n}},$ hence we get $\ell(s_{\overline{i+h}}s_{k_{1}}\cdots s_{k_{n}})< \ell(s_{k_{1}}\cdots s_{k_{n}}).$ By [Lu1, Proposition 1.7], we get that $s_{\overline{i+h}}s_{k_{1}}\cdots s_{k_{n}}=w'$ with $\ell(w')< n,$ and so $s_{k_{1}}\cdots s_{k_{n}}=s_{\overline{i+h}}w'.$

We need to check that $w'$ satisfies the equalities $\ell(s_{\overline{i+h}}w's_{j})=\ell(w')$ and $\ell(s_{\overline{i+h}}w')=\ell(w's_{j}).$ By definition, we have $w'=\rho^{-h}\cdot{}^{s_{i}}\!t^{-1}s_{i}\underline{w}$ and $\rho^{-h}\cdot{}^{s_{i}}\!t^{-1}s_{i}=\rho^{-h}s_{i}t^{-1}=s_{\overline{i+h}}\rho^{-h}t^{-1}.$ Thus, we get that
\begin{align*}
\ell(s_{\overline{i+h}}w's_{j})=\ell(s_{\overline{i+h}}\rho^{-h}\cdot{}^{s_{i}}\!t^{-1}s_{i}\underline{w}s_{j})=
\ell(\rho^{-h}t^{-1}\underline{w}s_{j})=\ell(\underline{w}s_{j})=\ell(s_{i}\underline{w})=\ell(w'),
\end{align*}
and
\begin{align*}
\ell(s_{\overline{i+h}}w')=\ell(s_{\overline{i+h}}\rho^{-h}\cdot{}^{s_{i}}\!t^{-1}s_{i}\underline{w})
=\ell(\rho^{-h}t^{-1}\underline{w})=\ell(\underline{w})=\ell(s_{i}\underline{w}s_{j})=\ell(w's_{j}).
\end{align*}

So by induction, we get that $s_{\overline{i+h}}w'e_{j}=e_{\overline{i+h}}w's_{j}$ and that
\begin{align*}
s_{i}\underline{w}e_{j}=s_{i}t\rho^{h}s_{\overline{i+h}}w'e_{j}={}^{s_{i}}\!t\rho^{h}s_{\overline{i+h}}e_{\overline{i+h}}w's_{j}
={}^{s_{i}}\!te_{i}\rho^{h}s_{\overline{i+h}}w's_{j}=e_{i}t\rho^{h}s_{\overline{i+h}}w's_{j}=e_{i}\underline{w}s_{j}.
\end{align*}

If $\ell(\underline{w})< \ell(s_{i}\underline{w}),$ then we have
\[\ell(s_{i}\cdot \underline{w}s_{j}\cdot s_{j})=\ell(s_{i}\underline{w})=\ell(\underline{w}s_{j}),\]
\[\ell(s_{i}\cdot \underline{w}s_{j})=\ell(\underline{w})=\ell(\underline{w}s_{j}\cdot s_{j}),\]
and
\[\ell(\underline{w}s_{j})=\ell(s_{i}\cdot \underline{w})> \ell(\underline{w})=\ell(s_{i}\cdot \underline{w}s_{j}).\]
Thus, we can apply the first case on the element $\underline{w}s_{j}$, and get that $s_{i}(\underline{w}s_{j})e_{j}=e_{i}(\underline{w}s_{j})s_{j},$ that is, $s_{i}\underline{w}e_{j}=e_{i}\underline{w}s_{j}.$

$(3)$ The equality follows directly from $(1)$ and $(2)$.
\end{proof}

The following proposition can be proved by a standard argument (see [Lu1, Proposition 3.3] for instance).
\begin{proposition}\label{coxeter-lemma-iso-pbwnbasis-4}
$\widehat{H}_{r,n}^{\mathrm{aff}}$ has an $\mathcal{R}$-basis consisting of the following elements$:$
\begin{align}\label{coxeter-lemma-iso-pbwnbasis-444}
\big\{T_{\underline{w}}\:|\:\underline{w}\in \widehat{W}_{r,n}\big\},
\end{align}
\[\text{ or }\qquad\big\{t_{1}^{\beta_{1}}\cdots t_{n}^{\beta_{n}}T_{\widehat{w}}\:|\:0\leq\beta_{1},\ldots,\beta_{n}\leq r-1,~\widehat{w}\in \widehat{W}\big\}.\]
\end{proposition}
\begin{proof}
From Definition \ref{Definition 2-4}, it is obvious that we have
\[t_{k}T_{\underline{w}}=T_{t_{k}\underline{w}}\qquad \text{ for any }1\leq k\leq n.\]
For any $0\leq i\leq n-1,$ we have
\begin{align*}
T_{s_{i}}T_{\underline{w}}=
\left\{
\begin{array}{ll}
T_{s_{i}\underline{w}}&\text{if $\ell(s_{i}\underline{w})=\ell(\underline{w})+1$},\\[0.3em]
T_{s_{i}\underline{w}}+(q-q^{-1})T_{e_{i}\underline{w}}&\text{if $\ell(s_{i}\underline{w})=\ell(\underline{w})-1$},
\end{array}
\right.
\end{align*}
and
\[T_{\rho}^{\pm 1}T_{\underline{w}}=T_{\rho^{\pm1}\underline{w}}.\]
Thus, the $\mathcal{R}$-submodule of $\widehat{H}_{r,n}^{\mathrm{aff}}$ generated by $\{T_{\underline{w}}\:|\:\underline{w}\in \widehat{W}_{r,n}\}$ is a left ideal of $\widehat{H}_{r,n}^{\mathrm{aff}}.$ Since it contains $1=T_{1}$, it is the whole algebra $\widehat{H}_{r,n}^{\mathrm{aff}}.$ In particular, $\widehat{H}_{r,n}^{\mathrm{aff}}$ is generated by the elements $\{T_{\underline{w}}\:|\:\underline{w}\in \widehat{W}_{r,n}\}.$

Next we prove that the set $\{T_{\underline{w}}\:|\:\underline{w}\in \widehat{W}_{r,n}\}$ is an $\mathcal{R}$-basis of $\widehat{H}_{r,n}^{\mathrm{aff}}.$ Consider the free $\mathcal{R}$-module $\mathcal{E}$ with bases $(e_{\underline{w}})_{\underline{w}\in \widehat{W}_{r,n}}.$ For each $1\leq k\leq n$ and $0\leq i\leq n-1,$ we define the following $\mathcal{R}$-linear maps from $\mathcal{E}$ to $\mathcal{E}$ by
\begin{align*}
P_{t_{k}}(e_{\underline{w}})&=e_{t_{k}\underline{w}},\\
P_{s_{i}}(e_{\underline{w}})&=
\left\{
\begin{array}{ll}
e_{s_{i}\underline{w}}&\text{if $\ell(s_{i}\underline{w})=\ell(\underline{w})+1$},\\[0.3em]
e_{s_{i}\underline{w}}+(q-q^{-1})e_{e_{i}\underline{w}}&\text{if $\ell(s_{i}\underline{w})=\ell(\underline{w})-1$},
\end{array}
\right.\\
P_{\rho}^{\pm 1}(e_{\underline{w}})&=e_{\rho^{\pm1}\underline{w}},
\end{align*}
and

\begin{align*}
Q_{t_{k}}(e_{\underline{w}})&=e_{\underline{w}t_{k}},\\
Q_{s_{i}}(e_{\underline{w}})&=
\left\{
\begin{array}{ll}
e_{\underline{w}s_{i}}&\text{if $\ell(\underline{w}s_{i})=\ell(\underline{w})+1$},\\[0.3em]
e_{\underline{w}s_{i}}+(q-q^{-1})e_{\underline{w}e_{i}}&\text{if $\ell(\underline{w}s_{i})=\ell(\underline{w})-1$},
\end{array}
\right.\\
Q_{\rho}^{\pm 1}(e_{\underline{w}})&=e_{\underline{w}\rho^{\pm1}}.
\end{align*}

Set $\widehat{S}^{\mathrm{aff}} :=\{t_{1},\ldots,t_{n}\}\cup \{\rho^{\pm 1}\}\cup S^{\mathrm{aff}}.$ We first prove the following claim:

$\bold{Claim(a)}$: $P_{u}Q_{v}=Q_{v}P_{u}$ for any $u ,v\in \widehat{S}^{\mathrm{aff}}.$

When either $u$ or $v$ belongs to the set $\{t_{1},\ldots,t_{n}\}\cup \{\rho^{\pm 1}\},$ it is easy to check that $\bold{Claim(a)}$ holds. Thus, it suffices to check that $P_{s_{i}}Q_{s_{j}}=Q_{s_{j}}P_{s_{i}}$ for any $s_{i}, s_{j}\in S^{\mathrm{aff}}.$ This can be proved by distinguishing the following six cases. Let $\underline{w}\in \widehat{W}_{r,n}.$

\noindent$\emph{Case}$ 1. $s_{i}\underline{w}s_{j},$ $s_{i}\underline{w},$ $\underline{w}s_{j},$ $\underline{w}$ have lengths $q+2,$ $q+1,$ $q+1,$ $q$. Then we have
\[P_{s_{i}}Q_{s_{j}}(e_{\underline{w}})=Q_{s_{j}}P_{s_{i}}(e_{\underline{w}})=e_{s_{i}\underline{w}s_{j}}.\]

\noindent$\emph{Case}$ 2. $\underline{w}$, $s_{i}\underline{w},$ $\underline{w}s_{j},$ $s_{i}\underline{w}s_{j}$ have lengths $q+2,$ $q+1,$ $q+1,$ $q$. Then we have
\begin{align*}
&P_{s_{i}}Q_{s_{j}}(e_{\underline{w}})=Q_{s_{j}}P_{s_{i}}(e_{\underline{w}})\\
&=e_{s_{i}\underline{w}s_{j}}+(q-q^{-1})e_{s_{i}\underline{w}e_{j}}+(q-q^{-1})e_{e_{i}\underline{w}s_{j}}+(q-q^{-1})^{2}e_{e_{i}\underline{w}e_{j}}.
\end{align*}

\noindent$\emph{Case}$ 3. $\underline{w}s_{j},$ $s_{i}\underline{w}s_{j}$, $\underline{w}$, $s_{i}\underline{w}$ have lengths $q+2,$ $q+1,$ $q+1,$ $q$. Then we have
\begin{align*}
P_{s_{i}}Q_{s_{j}}(e_{\underline{w}})=Q_{s_{j}}P_{s_{i}}(e_{\underline{w}})=
e_{s_{i}\underline{w}s_{j}}+(q-q^{-1})e_{e_{i}\underline{w}s_{j}}.
\end{align*}

\noindent$\emph{Case}$ 4. $s_{i}\underline{w},$ $s_{i}\underline{w}s_{j}$, $\underline{w}$, $\underline{w}s_{j}$ have lengths $q+2,$ $q+1,$ $q+1,$ $q$. Then we have
\begin{align*}
P_{s_{i}}Q_{s_{j}}(e_{\underline{w}})=Q_{s_{j}}P_{s_{i}}(e_{\underline{w}})=
e_{s_{i}\underline{w}s_{j}}+(q-q^{-1})e_{s_{i}\underline{w}e_{j}}.
\end{align*}

\noindent$\emph{Case}$ 5. $s_{i}\underline{w}s_{j}$, $\underline{w}$, $\underline{w}s_{j}$, $s_{i}\underline{w}$ have lengths $q+1,$ $q+1,$ $q,$ $q$. Then we have
\begin{align*}
P_{s_{i}}Q_{s_{j}}(e_{\underline{w}})=e_{s_{i}\underline{w}s_{j}}+(q-q^{-1})e_{s_{i}\underline{w}e_{j}}+(q-q^{-1})^{2}e_{e_{i}\underline{w}e_{j}},
\end{align*}
and
\begin{align*}
Q_{s_{j}}P_{s_{i}}(e_{\underline{w}})=
e_{s_{i}\underline{w}s_{j}}+(q-q^{-1})e_{e_{i}\underline{w}s_{j}}+(q-q^{-1})^{2}e_{e_{i}\underline{w}e_{j}}.
\end{align*}
By Lemma \ref{lemma-juyumaya-gene}(2), we have $s_{i}\underline{w}e_{j}=e_{i}\underline{w}s_{j}.$ Thus, we get that $P_{s_{i}}Q_{s_{j}}(e_{\underline{w}})=Q_{s_{j}}P_{s_{i}}(e_{\underline{w}}).$

\noindent$\emph{Case}$ 6. $s_{i}\underline{w}$, $\underline{w}s_{j}$, $\underline{w}$, $s_{i}\underline{w}s_{j}$ have lengths $q+1,$ $q+1,$ $q,$ $q$. Then we have
\begin{align*}
P_{s_{i}}Q_{s_{j}}(e_{\underline{w}})=e_{s_{i}\underline{w}s_{j}}+(q-q^{-1})e_{e_{i}\underline{w}s_{j}},
\end{align*}
and
\begin{align*}
Q_{s_{j}}P_{s_{i}}(e_{\underline{w}})=
e_{s_{i}\underline{w}s_{j}}+(q-q^{-1})e_{s_{i}\underline{w}e_{j}}.
\end{align*}
Hence, we also get that $P_{s_{i}}Q_{s_{j}}(e_{\underline{w}})=Q_{s_{j}}P_{s_{i}}(e_{\underline{w}})$ by Lemma \ref{lemma-juyumaya-gene}(2).

Thus, we have proved the $\bold{Claim(a)}$. Then we can repeat the arguments as done in the proof of [Lu1, Proposition 3.3] to conclude that the set $\{T_{\underline{w}}\:|\:\underline{w}\in \widehat{W}_{r,n}\}$ is an $\mathcal{R}$-basis of $\widehat{H}_{r,n}^{\mathrm{aff}}.$ Since the procedure is routine, we shall skip the details.
\end{proof}

Set $S^{\mathrm{aff}}(1) :=\{ts\:|\:t\in \mathcal{T}\text{ and }s\in S^{\mathrm{aff}}\}.$ Take $q_{ts_{i}} :=1$ and $c_{ts_{i}} :=(q-q^{-1})te_{i}$ for all $0\leq i\leq n-1$ and $t\in \mathcal{T}.$ It has been proved in [ChS, Section 4] that $q_{ts_{i}}$'s and $c_{ts_{i}}$'s satisfy the conditions [Vi2, Theorem 2.4(A)(1)-(2)].

Multiplying two sides of \eqref{rel-def-Heckealg8} by $T_{t}T_{{}^{s_{i}}\!t}$ and noting that $T_{{}^{s_{i}}\!t}e_{i}=e_{i}T_{t}$, we get that
\begin{align}\label{main-theorem-iden24}
T_{ts_{i}}^{2}=(ts_{i})^{2}+c_{ts_{i}}T_{ts_{i}} \quad \text{ for all }0\leq i\leq n-1 \text{ and }t\in \mathcal{T}.
\end{align}
Hence, from the presentation of $\widehat{H}_{r,n}^{\mathrm{aff}}$ given in Definition 2.4 and \eqref{main-theorem-iden24}, we see that the $\mathcal{R}$-bases $\{T_{\underline{w}}\:|\:\underline{w}\in \widehat{W}_{r,n}\}$ of $\widehat{H}_{r,n}^{\mathrm{aff}}$ satisfy the braid and quadratic relations in [Vi2, Theorem 2.4(B)].

Thus, from Theorem \ref{iwahori-matsumoto-pre} we immediately get the following corollary, which was previously proved in [ChS, Theorem 4.1].
\begin{corollary}
The affine Yokonuma-Hecke algebra $\widehat{Y}_{r,n}$ is a particular case of the pro-$p$-Iwahori-Hecke algebras.
\end{corollary}

\subsection{One application}

let $\mathbb{K}$ be an algebraically closed field of characteristic $p$ such that $p$ does not divide $r.$ In this subsection, we shall consider the specializations over $\mathbb{K}$ of various algebras $\widehat{Y}_{r,n},$ $\widehat{H}_{r,n}^{\mathrm{aff}}$, and so on; moreover, we shall denote the specialization algebras with the same symbols.

We first review some constructions presented in [JaPA, Section 2]. Assume that $\{\zeta_1,\ldots,\zeta_{r}\}$ is the set of all $r$-th roots of unity. A character $\chi$ of $\mathcal{T}$ over $\mathbb{K}$ is determined by the values $\chi(t_{j})\in \{\zeta_1,\ldots,\zeta_{r}\}$ for $1\leq j\leq n.$ We denote by $\text{Irr}(\mathcal{T})$ the set of characters of $\mathcal{T}$ over $\mathbb{K}$.

By Lemma \ref{coxeter-lemma-iso1}, we will identify $\widehat{W}$ with $\mathfrak{X}\rtimes \mathfrak{S}_{n}.$ We have a natural group homomorphism $\sigma$ from $\widehat{W}$ to $\mathfrak{S}_{n},$ which is defined by
\[\sigma(X_{j})=1\quad \text{ for }1\leq j\leq n\quad\text{and}\quad \sigma(s_{i})=s_{i}\text{ for }1\leq i\leq n-1.\]
Moreover, we have an action of $\mathfrak{S}_{n}$ on $\mathcal{T}$ by permutations, which in turn induces an action of $\mathfrak{S}_{n}$ on $\text{Irr}(\mathcal{T})$ given by
\[w(\chi)(t_i)=\chi(t_{w^{-1}(i)})\quad \text{ for all }w\in \mathfrak{S}_{n}, \chi\in \text{Irr}(\mathcal{T})\text{ and }1\leq i\leq n.\]
Thus, we get an action of $\widehat{W}$ on $\text{Irr}(\mathcal{T})$ by composing $\sigma$ and the action of $\mathfrak{S}_{n}$ on $\text{Irr}(\mathcal{T})$ defined above.

For each $\chi\in \text{Irr}(\mathcal{T})$, the primitive idempotent $E_{\chi}$ of $\mathcal{T}$ associated to $\chi$ can be explicitly written as follows:
\begin{equation}\label{idempotents-cuiiuc}
E_{\chi}=\prod_{1\leq i\leq n}\bigg(\frac{1}{r}\sum_{0\leq s\leq r-1}\chi(t_{i})^{s}t_{i}^{-s}\bigg).\end{equation}
Then, the set $\{E_{\chi}\:|\:\chi\in \text{Irr}(\mathcal{T})\}$ forms a complete set of orthogonal idempotents, and is a $\mathbb{K}$-basis of $\mathbb{K}\mathcal{T},$ where we identify the group algebra $\mathbb{K}\mathcal{T}$ of $\mathcal{T}$ over $\mathbb{K}$ with the subalgebra of $\widehat{H}_{r,n}^{\mathrm{aff}}$ generated by $t_{1},\ldots,t_{n}.$

\begin{lemma}\label{another-basis-cui}
The elements $\{E_{\chi}T_{\widehat{w}}\:|\:\chi\in \mathrm{Irr}(\mathcal{T})\text{ and }\widehat{w}\in \widehat{W}\}$ is a $\mathbb{K}$-basis of $\widehat{H}_{r,n}^{\mathrm{aff}}.$
\end{lemma}
\begin{proof}
By Proposition \ref{coxeter-lemma-iso-pbwnbasis-4} and the claims above, we see that $\widehat{H}_{r,n}^{\mathrm{aff}}$ is generated by the elements $\{E_{\chi}T_{\widehat{w}}\}.$ Thus, it suffices to prove that they are linearly independent. If they are linearly dependent, that is, there exist some $a_{ij}\in \mathbb{K}$ such that
\begin{align}\label{linear-indepen}
\sum_{i,j}a_{ij}E_{\chi_{i}}T_{\widehat{w}_{j}}=0,
\end{align}
where $a_{ij}$ are not all zero. We might as well assume that $a_{11}\neq 0.$ Multiplying two sides of (\ref{linear-indepen}) by $E_{\chi_{1}}$, we get that $\sum_{j}a_{1j}E_{\chi_{1}}T_{\widehat{w}_{j}}=0.$ But $E_{\chi_{1}}$ can be written as a linear combination of some $t_{1}^{\alpha_{1}}\cdots t_{n}^{\alpha_{n}}$'s. Thus, by the equality above we can easily get that the elements $\{t_{1}^{\beta_{1}}\cdots t_{n}^{\beta_{n}}T_{\widehat{w}}\}$ are linearly dependent. This is a contradiction. We are done.
\end{proof}

An $r$-composition of $n,$ denoted by $\mu \models n$, is an $r$-tuple $\mu=(\mu_1,\ldots,\mu_{r})\in \mathbb{Z}_{\geq 0}^{r}$ such that $\Sigma_{1\leq a\leq r}\mu_{a}=n.$ Let $\mathcal{C}_{r,n}$ be the set of $r$-compositions of $n.$ Assume that $\chi\in \text{Irr}(\mathcal{T}).$ For $a\in \{1,\ldots,r\},$ let $\mu_a$ be the cardinal of elements $j\in \{1,\ldots,n\}$ such that $\chi(t_j)=\zeta_a.$ Then the sequence $(\mu_1,\ldots,\mu_r)\in \mathcal{C}_{r,n}$, and we denote it by $\text{Comp}(\chi).$

For each $\mu \models n,$ we define a particular character $\chi_{1}^{\mu}\in \text{Irr}(\mathcal{T})$ by
\begin{equation}\label{chi1-mu-cui}
\left\{\begin{array}{ccccccc}
\chi_1^{\mu} (t_1)&=&\ldots & =& \chi_1^{\mu} (t_{\mu_1})&=& \zeta_1\ ,\\[0.2em]
\chi_1^{\mu} (t_{\mu_1+1})&=&\ldots & =& \chi_1^{\mu} (t_{\mu_1+\mu_2})&=& \zeta_2\ ,\\
\vdots &\vdots &\vdots &\vdots &\vdots &\vdots &\vdots  \\
\chi_1^{\mu} (t_{\mu_1+\dots+\mu_{r-1}+1})&=&\ldots & =& \chi_1^{\mu} (t_{n})&=& \zeta_r\ .\\
\end{array}\right.
\end{equation}
Notice that $\text{Comp}(\chi_{1}^{\mu})=\mu.$ From \eqref{chi1-mu-cui}, we see that the stabilizer of $\chi_{1}^{\mu}$ under the action of $\mathfrak{S}_{n}$ is the Young subgroup $\mathfrak{S}^{\mu},$ which is defined to be $\mathfrak{S}_{\mu_1}\times\cdots\times\mathfrak{S}_{\mu_r}$. Notice that there is a unique representative of minimal length in each left coset in $\mathfrak{S}_{n}/\mathfrak{S}^{\mu}.$ We shall denote these distinguished left coset representatives by $\{\pi_{1,\mu},\ldots,\pi_{m_{\mu}, \mu}\}$ by the convention that $\pi_{1,\mu}=1$ and set $\chi_{k}^{\mu} :=\pi_{k,\mu}(\chi_{1}^{\mu})$ for $1\leq k\leq m_{\mu}.$

For each $\mu \models n,$ we set
\begin{equation*}
E_{\mu} :=\sum_{\text{Comp}(\chi)=\mu}E_{\chi}.
\end{equation*}
Then the set $\{E_{\mu}\:|\:\mu\in \mathcal{C}_{r,n}\}$ forms a complete set of pairwise orthogonal central idempotents in $\widehat{H}_{r,n}^{\mathrm{aff}}$. In particular, we have the following decomposition of $\widehat{H}_{r,n}^{\mathrm{aff}}$ into a direct sum of two-sided ideals:
\begin{equation}\label{direct-sum-cui}
\widehat{H}_{r,n}^{\mathrm{aff}}=\bigoplus_{\text{Comp}(\chi)=\mu}E_{\mu}\widehat{H}_{r,n}^{\mathrm{aff}}.
\end{equation}
Moreover, the following elements
\[\{E_{\chi_{k}^{\mu}}T_{\widehat{w}}\:|\:1\leq k\leq m_{\mu}\text{ and }\widehat{w}\in \widehat{W}\}\]
form a $\mathbb{K}$-basis of $E_{\mu}\widehat{H}_{r,n}^{\mathrm{aff}}.$

Let $\widehat{\mathcal{H}}_{n}$ denote the extended affine Hecke algebra of type $A$ associated to $\widehat{W}$ over $\mathbb{K}$, which is endowed with a standard basis $\{S_{\widehat{w}}\:|\:\widehat{w}\in \widehat{W}\}$. For each $\mu \models n,$ we denote by $\widehat{H}^{\mu}$ the $\mathbb{K}$-subalgebra of $\widehat{H}_{n}$ generated by the elements $\{S_{\widehat{w}}\:|\:\widehat{w}\in \widehat{\mathfrak{S}^{\mu}}\}$, where $\widehat{\mathfrak{S}^{\mu}} :=\mathbb{Z}^{n}\rtimes \mathfrak{S}^{\mu}$ is the subgroup of $\widehat{W},$ which is exactly the stabilizer of $\chi_{1}^{\mu}$ under the action of $\widehat{W}.$ The algebra $\widehat{H}^{\mu}$ is naturally isomorphic to $\widehat{\mathcal{H}}_{\mu_{1}}\otimes \cdots \otimes\widehat{\mathcal{H}}_{\mu_{r}}.$

The following lemma can easily be proved by a direct calculation.
\begin{lemma}
Let $\mu \models n.$ There exists an algebra isomorphism \[\phi_{\mu}: \widehat{H}^{\mu}\overset{\sim}{\longrightarrow} E_{\chi_{1}^{\mu}}\widehat{H}_{r,n}^{\mathrm{aff}}E_{\chi_{1}^{\mu}}, \] which is defined by $\phi_{\mu}(S_{\widehat{w}})=E_{\chi_{1}^{\mu}}T_{\widehat{w}}E_{\chi_{1}^{\mu}}$ for any $\widehat{w}\in \widehat{\mathfrak{S}^{\mu}}.$
\end{lemma}

For each $\mu \models n,$ let $\text{Mat}_{m_{\mu}}(\widehat{H}^{\mu})$ be the algebra of matrices of size $m_{\mu}$ with coefficients in $\widehat{H}^{\mu}.$ We define a linear map
\begin{equation*}
\Phi_{\mu}: E_{\mu}\widehat{H}_{r,n}^{\mathrm{aff}}\rightarrow \text{Mat}_{m_{\mu}}(\widehat{H}^{\mu})
\end{equation*}
by
\begin{equation}\label{Phi-mu-cui}
\Phi_{\mu}(E_{\chi_{k}^{\mu}}T_{\widehat{w}})=S_{\pi_{k,\mu}^{-1}\widehat{w}\pi_{j,\mu}}M_{k,j},
\end{equation}
where $j\in \{1,\ldots,m_{\mu}\}$ is the unique number such that $\widehat{w}(\chi_{j}^{\mu})=\chi_{k}^{\mu}$ for given $k,$ and $M_{k,j}$ denotes the elementary matrix with $1$ in the position $(k,j).$

We also define a linear map
\begin{equation*}
\Psi_{\mu}: \text{Mat}_{m_{\mu}}(\widehat{H}^{\mu})\rightarrow E_{\mu}\widehat{H}_{r,n}^{\mathrm{aff}}
\end{equation*}
by
\begin{equation}\label{Psi-mu-cui}
\Psi_{\mu}((S_{\widehat{w}_{i,j}})_{1\leq i,j\leq m_{\mu}})=\sum_{1\leq i,j\leq m_{\mu}}E_{\chi_{i}^{\mu}}T_{\pi_{i,\mu}\widehat{w}_{i,j}\pi_{j,\mu}^{-1}}E_{\chi_{j}^{\mu}}
\end{equation}
for $\widehat{w}_{i,j}\in \widehat{\mathfrak{S}^{\mu}}.$

We define the linear maps $\Phi_{r, n} :=\bigoplus_{\mu\in \mathcal{C}_{r,n}}\Phi_{\mu}$ and $\Psi_{r, n} :=\bigoplus_{\mu\in \mathcal{C}_{r,n}}\Psi_{\mu}.$ The following theorem can be proved in exactly the same way as in [JaPA, Theorem 3.1], and we skip the details.
\begin{theorem}\label{isomorphsim-theorem1-cui}
For $\mu \models n,$ the linear map $\Phi_{\mu}$ is an isomorphism of algebras with the inverse map $\Psi_{\mu}.$ Accordingly, $\Phi_{r, n}$ and $\Psi_{r, n}$ establish an isomorphism of algebras between $\widehat{H}_{r,n}^{\mathrm{aff}}$ and $\bigoplus_{\mu\in \mathcal{C}_{r,n}}\emph{Mat}_{m_{\mu}}(\widehat{H}^{\mu}).$
\end{theorem}

Combining Theorems \ref{iwahori-matsumoto-pre} and \ref{isomorphsim-theorem1-cui}, we immediately get the following result, which was previously proved in [C1, Theorem 5.1] and also [PA, Theorem 3.1].
\begin{theorem}\label{isomorphsim-theorem11-cuiiuc}
There is a canonical isomorphism between the affine Yokonuma-Hecke algebra $\widehat{Y}_{r,n}$ and $\bigoplus_{\mu\in \mathcal{C}_{r,n}}\emph{Mat}_{m_{\mu}}(\widehat{H}^{\mu}).$
\end{theorem}

\section{A third presentation}

Let $\zeta=e^{2\pi i/r}$ and let $A$ be the square matrix of degree $r$ whose $ij$-entry is equal to $a_{ij}=\zeta^{j(i-1)}$ for $1\leq i, j\leq r;$ i.e., $A$ is the usual Vandermonde matrix. Let $\Delta=\mathrm{det} A$ is the Vandermonde determinant, that is, $\Delta=\prod_{1\leq j< i\leq r}(\zeta^{i}-\zeta^{j}).$ We can write the inverse of $A$ as $A^{-1}=\Delta^{-1}B,$ where $B=(b_{ij}(\zeta))$ is the adjoint matrix of $A$.

For each $1\leq i\leq r,$ we define a polynomial $F_{i}(X)\in \mathbb{Z}[\zeta][X]$ by
\[F_{i}(X) :=\sum_{1\leq j\leq r}b_{ij}(\zeta)X^{j-1}.\]
Let $\mathfrak{R}=\mathbb{Z}[q,q^{-1},\zeta, \Delta^{-1}],$ where $q$ is an indeterminate. We first give the definition of an $\mathfrak{R}$-associative algebra $\widehat{\mathcal{C}}_{r,n}^{\mathrm{aff}}.$

\begin{definition}\label{Definition 3-1}
We define an $\mathfrak{R}$-associative algebra $\widehat{\mathcal{C}}_{r,n}^{\mathrm{aff}},$ which is generated by the elements $w_{1},\ldots,w_{n},$ $h_{s_{0}},\ldots,h_{s_{n-1}},$ $h_{\rho}^{\pm 1}$ with the following relations:
\begin{align}
w_i^r&=1 \quad \qquad\qquad\qquad\qquad\mbox{for all $1\leq i\leq n$;}\label{athid-rel-def-Heckealg1}\\[0.1em]
w_iw_j&=w_jw_i \qquad\qquad\qquad\quad\hspace{1.4mm}  \mbox{for all $1\leq i, j\leq n$;}\label{athid-rel-def-Heckealg2}\\[0.1em]
h_{\rho}w_j&=w_{j-1}h_{\rho} \quad \quad\qquad\qquad\hspace{1.5mm}\mbox{for all $1\leq j\leq n$;}\label{athid-rel-def-Heckealg3}\\[0.1em]
h_{\rho}h_{s_{\overline{i}}}&=h_{s_{\overline{i-1}}}h_{\rho}  \quad \quad\qquad\qquad\hspace{1.5mm}\mbox{for all $0\leq i\leq n-1 $;}\label{athid-rel-def-Heckealg4}\\[0.1em]
h_{s_{i}}h_{s_{j}}&=h_{s_{j}}h_{s_{i}}\quad \quad\qquad\qquad\hspace{3.7mm}  \text{if $i-j\not\equiv \pm 1$ (mod $n$)};\label{athid-rel-def-Heckealg5}\\[0.1em]
h_{s_{\overline{i}}}h_{s_{\overline{i+1}}}h_{s_{\overline{i}}}&=h_{s_{\overline{i+1}}}h_{s_{\overline{i}}}h_{s_{\overline{i+1}}}\quad \quad\qquad \hspace{0.3mm}  \text{if $0\leq i\leq n-1$ and $n\geq 3$};\label{athid-rel-def-Heckealg6}\\[0.1em]
h_{s_{i}}^{2}&=1+(q-q^{-1})h_{s_{i}} \quad \hspace{6mm}\mbox{for all $0\leq i\leq n-1$;}\label{athid-rel-def-Heckealg7}\\[0.1em]
h_{s_{i}}w_i&=w_{i+1}h_{s_{i}}-\Delta^{-2}\sum_{c_{1}< c_{2}}(\zeta^{c_2}-\zeta^{c_{1}})(q-q^{-1})F_{c_1}(w_{i})F_{c_2}(w_{i+1}) \hspace{2.5mm}\mbox{for $1\leq i\leq n-1$;}\label{athid-rel-def-Heckealg8}\\[0.1em]
h_{s_{i}}w_{i+1}&=w_{i}h_{s_{i}}+\Delta^{-2}\sum_{c_{1}< c_{2}}(\zeta^{c_2}-\zeta^{c_{1}})(q-q^{-1})F_{c_1}(w_{i})F_{c_2}(w_{i+1}) \hspace{2.5mm}\mbox{for $1\leq i\leq n-1$;}\label{athid-rel-def-Heckealg9}\\[0.1em]
h_{s_{i}}w_{l}&=w_{l}h_{s_{i}} \quad \qquad\qquad\qquad\hspace{1.5mm}\mbox{for all $l\neq i, i+1$ and $1\leq i\leq n-1$;}\label{athid-rel-def-Heckealg10}\\[0.1em]
h_{\rho}h_{\rho}^{-1}&=h_{\rho}^{-1}h_{\rho}=1,\label{athid-rel-def-Heckealg11-com}
\end{align}
where $w_{0} :=w_{n}$ and in the expressions above, the sum is taken over all $1\leq c_{1}, c_{2}\leq r$ such that $c_{1}< c_{2}.$
\end{definition}

Let $R=\mathbb{Z}[\frac{1}{r}][q,q^{-1},\zeta, \Delta^{-1}].$ We extend the algebras $\widehat{Y}_{r,n},$ $\widehat{H}_{r,n}^{\mathrm{aff}}$ and $\widehat{\mathcal{C}}_{r,n}^{\mathrm{aff}}$ from $\mathcal{R}$ and $\mathfrak{R}$ to $R$, respectively. We shall denote the extension algebras by the same notations such that $\widehat{Y}_{r,n},$ $\widehat{H}_{r,n}^{\mathrm{aff}}$ and $\widehat{\mathcal{C}}_{r,n}^{\mathrm{aff}}$ are all defined on $R$ in the rest of this section.

We now state the main result of this section.
\begin{theorem}\label{iwahori-matsumoto-pre-cuimaintheorem}
We have an $R$-algebra isomorphism $\phi :\widehat{H}_{r,n}^{\mathrm{aff}}\rightarrow \widehat{\mathcal{C}}_{r,n}^{\mathrm{aff}}$ given as follows$:$

\noindent for $1\leq j\leq n$,
\[\phi(t_{j})=w_{j},\]
\noindent for $0\leq i\leq n-1$,
\[\phi(T_{s_{i}})=h_{s_{i}}-\Delta^{-2}(q-q^{-1})\sum_{c_{1}< c_{2}}F_{c_1}(w_{i})F_{c_2}(w_{i+1}),\]
\noindent and
\[\phi(T_{\rho}^{\pm1})=h_{\rho}^{\pm1}.\]
Moreover, its inverse $\psi :\widehat{\mathcal{C}}_{r,n}^{\mathrm{aff}}\rightarrow \widehat{H}_{r,n}^{\mathrm{aff}}$ is defined as follows$:$

\noindent for $1\leq j\leq n$,
\[\psi(w_{j})=t_{j},\]
\noindent for $0\leq i\leq n-1$,
\[\psi(h_{s_{i}})=T_{s_{i}}+\Delta^{-2}(q-q^{-1})\sum_{c_{1}< c_{2}}F_{c_1}(t_{i})F_{c_2}(t_{i+1}),\]
\noindent and
\[\psi(h_{\rho}^{\pm1})=T_{\rho}^{\pm1}.\]
\end{theorem}

We denote by $\mathcal{H}_{r,n}$ the subalgebra of $\widehat{H}_{r,n}^{\mathrm{aff}}$ generated by the elements $t_{1},\ldots, t_{n},$ $T_{s_{1}},\ldots,$ $T_{s_{n-1}},$ which is canonically isomorphic to the Yokonuma-Hecke algebra defined in [ChPA, Section 2.1], and we shall not distinguish between them.

We denote by $\mathcal{C}_{r,n}$ the subalgebra of $\widehat{\mathcal{C}}_{r,n}^{\mathrm{aff}}$ generated by the elements $w_{1},\ldots, w_{n},$ $h_{s_{1}},\ldots,$ $h_{s_{n-1}},$ which is canonically isomorphic to the modified Ariki-Koike algebra defined in [S, Section 3.6] with $u_{i}=\zeta^{i}$ for $1\leq i\leq r$, and we shall not distinguish between them.

The following proposition has been proved in [ER, Theorem 7].
\begin{proposition}\label{er-isomorphism7}
There is a canonical isomorphism between $\mathcal{H}_{r,n}$ and $\mathcal{C}_{r,n}$, which is explicitly described by $\varphi :\mathcal{H}_{r,n}\rightarrow \mathcal{C}_{r,n}$ and $\chi :\mathcal{C}_{r,n}\rightarrow \mathcal{H}_{r,n}$, where $\varphi$ is defined as follows$:$

\noindent for $1\leq j\leq n$,
\[\varphi(t_{j})=w_{j},\]
\noindent for $1\leq i\leq n-1$,
\[\varphi(T_{s_{i}})=h_{s_{i}}-\Delta^{-2}(q-q^{-1})\sum_{c_{1}< c_{2}}F_{c_1}(w_{i})F_{c_2}(w_{i+1}),\]
\noindent and its inverse $\chi$ is defined as follows$:$

\noindent for $1\leq j\leq n$,
\[\chi(w_{j})=t_{j},\]
\noindent for $1\leq i\leq n-1$,
\[\chi(h_{s_{i}})=T_{s_{i}}+\Delta^{-2}(q-q^{-1})\sum_{c_{1}< c_{2}}F_{c_1}(t_{i})F_{c_2}(t_{i+1}).\]
\end{proposition}

We then give the definition of two algebras as follows.
\begin{definition}\label{Definitionisomor-3-4}
We denote by $\mathcal{H}_{r,n}^{1}$ the subalgebra of $\widehat{H}_{r,n}^{\mathrm{aff}}$ generated by the elements $t_{1},\ldots, t_{n},$ $T_{s_{0}},\ldots,$ $T_{s_{n-2}},$ whose explicit presentation is as follows:
\begin{align}
t_i^r&=1 \quad \qquad\qquad\qquad\qquad\mbox{for all $1\leq i\leq n$;}\label{rel-def-Heckealgadd1}\\[0.1em]
t_it_j&=t_jt_i \qquad\qquad\qquad\qquad\hspace{0.4mm}  \mbox{for all $1\leq i, j\leq n$;}\label{rel-def-Heckealgadd2}\\[0.1em]
T_{s_{0}}t_1&=t_{n}T_{s_{0}} \quad \quad\qquad\qquad\hspace{1mm}\label{rel-def-Heckealgadd3}\\[0.1em]
T_{s_{0}}t_n&=t_{1}T_{s_{0}} \quad \quad\qquad\qquad\hspace{1mm}\label{rel-def-Heckealgaddadd3}\\[0.1em]
T_{s_{0}}t_{k}&=t_{k}T_{s_{0}} \quad \qquad\qquad\qquad\hspace{1.5mm}\mbox{for all $2\leq k\leq n-1$;}\label{rel-def-Heckealgaddaddadd3}\\[0.1em]
T_{s_{i}}t_j&=t_{s_i(j)}T_{s_{i}} \quad \quad\qquad\qquad\hspace{1.2mm}\mbox{for all $1\leq i\leq n-2$ and $1\leq j\leq n$;}\label{rel-def-Heckealgassadd3}\\[0.1em]
T_{s_{i}}T_{s_{j}}&=T_{s_{j}}T_{s_{i}}\quad \quad\qquad\qquad\hspace{3.95mm}  \text{if $|i-j |> 1$};\label{rel-def-Heckealgadd6}\\[0.1em]
T_{s_{i}}T_{s_{i+1}}T_{s_{i}}&=T_{s_{i+1}}T_{s_{i}}T_{s_{i+1}}\quad \quad\qquad \hspace{0.55mm}  \text{if $0\leq i\leq n-3$ and $n\geq 3$};\label{rel-def-Heckealgadd7}\\[0.1em]
T_{s_{i}}^{2}&=1+(q-q^{-1})e_{i}T_{s_{i}} \quad \hspace{3.65mm}\mbox{for all $0\leq i\leq n-2$,}\label{rel-def-Heckealgadd8}
\end{align}
where $t_{0} :=t_{n}$ and for each $0\leq i\leq n-2$,
$$e_{i} :=\frac{1}{r}\sum\limits_{s=0}^{r-1}t_{i}^{s}t_{i+1}^{-s}.$$
\end{definition}

\begin{definition}\label{Definitionisomorphism-3-5}
We denote by $\mathcal{C}_{r,n}^{1}$ the subalgebra of $\widehat{\mathcal{C}}_{r,n}^{\mathrm{aff}}$ generated by the elements $w_{1},\ldots, w_{n},$ $h_{s_{0}},\ldots,$ $h_{s_{n-2}},$ whose explicit presentation is as follows:
\begin{align}
w_i^r&=1 \quad \qquad\qquad\qquad\qquad\mbox{for all $1\leq i\leq n$;}\label{athid-rel-def-Heckealgass1}\\[0.1em]
w_iw_j&=w_jw_i \qquad\qquad\qquad\quad\hspace{1.4mm}  \mbox{for all $1\leq i, j\leq n$;}\label{athid-rel-def-Heckealgass2}\\[0.1em]
h_{s_{i}}h_{s_{j}}&=h_{s_{j}}h_{s_{i}}\quad \quad\qquad\qquad\hspace{3.7mm}  \text{if $|i-j |> 1$};\label{athid-rel-def-Heckealgass5}\\[0.1em]
h_{s_{i}}h_{s_{i+1}}h_{s_{i}}&=h_{s_{i+1}}h_{s_{i}}h_{s_{i+1}}\quad \quad\qquad \hspace{0.3mm}  \text{if $0\leq i\leq n-3$ and $n\geq 3$};\label{athid-rel-def-Heckealgass6}\\[0.1em]
h_{s_{i}}^{2}&=1+(q-q^{-1})h_{s_{i}} \quad \hspace{6mm}\mbox{for all $0\leq i\leq n-2$;}\label{athid-rel-def-Heckealgass7}\\[0.1em]
h_{s_{i}}w_i&=w_{i+1}h_{s_{i}}-\Delta^{-2}\sum_{c_{1}< c_{2}}(\zeta^{c_2}-\zeta^{c_{1}})(q-q^{-1})F_{c_1}(w_{i})F_{c_2}(w_{i+1}) \hspace{2.5mm}\mbox{for $0\leq i\leq n-2$;}\label{athid-rel-def-Heckealgass8}\\[0.1em]
h_{s_{i}}w_{i+1}&=w_{i}h_{s_{i}}+\Delta^{-2}\sum_{c_{1}< c_{2}}(\zeta^{c_2}-\zeta^{c_{1}})(q-q^{-1})F_{c_1}(w_{i})F_{c_2}(w_{i+1}) \hspace{2.5mm}\mbox{for $0\leq i\leq n-2$;}\label{athid-rel-def-Heckealgass9}\\[0.1em]
h_{s_{i}}w_{l}&=w_{l}h_{s_{i}} \quad \qquad\qquad\qquad\hspace{1.5mm}\mbox{for all $l\not\equiv i, i+1$ (mod $n$) and $0\leq i\leq n-2$,}\label{athid-rel-def-Heckealgass10}
\end{align}
where $w_{0} :=w_{n}$ and in the expressions above, the sum is taken over all $1\leq c_{1}, c_{2}\leq r$ such that $c_{1}< c_{2}.$
\end{definition}

By using Proposition \ref{er-isomorphism7}, we can easily get the following result.
\begin{proposition}\label{ERer-isomorphism8}
There is a canonical isomorphism between $\mathcal{H}_{r,n}^{1}$ and $\mathcal{C}_{r,n}^{1}$, which is explicitly described by $\phi_{1} :\mathcal{H}_{r,n}^{1}\rightarrow \mathcal{C}_{r,n}^{1}$ and $\psi_{1} :\mathcal{C}_{r,n}^{1}\rightarrow \mathcal{H}_{r,n}^{1}$, where $\phi_{1}$ is defined as follows$:$

\noindent for $1\leq j\leq n$,
\[\phi_{1}(t_{j})=w_{j},\]
\noindent for $0\leq i\leq n-2$,
\[\phi_{1}(T_{s_{i}})=h_{s_{i}}-\Delta^{-2}(q-q^{-1})\sum_{c_{1}< c_{2}}F_{c_1}(w_{i})F_{c_2}(w_{i+1}),\]
\noindent and its inverse $\psi_{1}$ is defined as follows$:$

\noindent for $1\leq j\leq n$,
\[\psi_{1}(w_{j})=t_{j},\]
\noindent for $0\leq i\leq n-2$,
\[\psi_{1}(h_{s_{i}})=T_{s_{i}}+\Delta^{-2}(q-q^{-1})\sum_{c_{1}< c_{2}}F_{c_1}(t_{i})F_{c_2}(t_{i+1}).\]
\end{proposition}

We further define the following two algebras.
\begin{definition}\label{Definitionisomoraddass-3-4}
We denote by $\mathcal{H}_{r,n}^{2}$ the subalgebra of $\widehat{H}_{r,n}^{\mathrm{aff}}$ generated by the elements $t_{1},\ldots, t_{n},$ $T_{s_{0}},T_{s_{2}}\ldots,$ $T_{s_{n-1}},$ whose explicit presentation is as follows:
\begin{align}
t_i^r&=1 \quad \qquad\qquad\qquad\qquad\mbox{for all $1\leq i\leq n$;}\label{rel-def-Heckealgaddcui1}\\[0.1em]
t_it_j&=t_jt_i \qquad\qquad\qquad\qquad\hspace{0.4mm}  \mbox{for all $1\leq i, j\leq n$;}\label{rel-def-Heckealgaddcui2}\\[0.1em]
T_{s_{0}}t_1&=t_{n}T_{s_{0}} \quad \quad\qquad\qquad\hspace{1mm}\label{rel-def-Heckealgaddcui3}\\[0.1em]
T_{s_{0}}t_n&=t_{1}T_{s_{0}} \quad \quad\qquad\qquad\hspace{1mm}\label{rel-def-Heckealgaddaddcui3}\\[0.1em]
T_{s_{0}}t_{k}&=t_{k}T_{s_{0}} \quad \qquad\qquad\qquad\hspace{1.5mm}\mbox{for all $2\leq k\leq n-1$;}\label{rel-def-Heckealgaddaddaddcui3}\\[0.1em]
T_{s_{i}}t_j&=t_{s_i(j)}T_{s_{i}} \quad \quad\qquad\qquad\hspace{1.2mm}\mbox{for all $2\leq i\leq n-1$ and $1\leq j\leq n$;}\label{rel-def-Heckealgassaddcui3}\\[0.1em]
T_{s_{i}}T_{s_{j}}&=T_{s_{j}}T_{s_{i}}\quad \quad\qquad\qquad\hspace{3.95mm}  \text{if $|i-j |> 1$ and $2\leq i, j\leq n-1$};\label{rel-def-Heckealgaddcui6}\\[0.1em]
T_{s_{i}}T_{s_{i+1}}T_{s_{i}}&=T_{s_{i+1}}T_{s_{i}}T_{s_{i+1}}\quad \quad\qquad \hspace{0.55mm}  \text{if $2\leq i\leq n-2$ and $n\geq 3$};\label{rel-def-Heckealgaddcui7}\\[0.1em]
T_{s_{0}}T_{s_{k}}&=T_{s_{k}}T_{s_{0}}\quad \quad\qquad\qquad\hspace{3.7mm}  \text{for all $2\leq k\leq n-2$};\label{reladd-def-Heckealgaddcuiass6}\\[0.1em]
T_{s_{0}}T_{s_{n-1}}T_{s_{0}}&=T_{s_{n-1}}T_{s_{0}}T_{s_{n-1}}\quad \quad\qquad \hspace{0.55mm}\label{relassadd-def-Heckealgaddcui7}\\[0.1em]
T_{s_{i}}^{2}&=1+(q-q^{-1})e_{i}T_{s_{i}} \quad \hspace{3.8mm}\mbox{for all $i=0$ and $2\leq i\leq n-1$,}\label{rel-def-Heckealgaddcui8}
\end{align}
where $t_{0} :=t_{n}$ and for each $i=0$ and $2\leq i\leq n-1$,
$$e_{i} :=\frac{1}{r}\sum\limits_{s=0}^{r-1}t_{i}^{s}t_{i+1}^{-s}.$$
\end{definition}

\begin{definition}\label{Definitionisomorphismaddcui-3-5}
We denote by $\mathcal{C}_{r,n}^{2}$ the subalgebra of $\widehat{\mathcal{C}}_{r,n}^{\mathrm{aff}}$ generated by the elements $w_{1},\ldots, w_{n},$ $h_{s_{0}}, h_{s_{2}}\ldots,$ $h_{s_{n-1}},$ whose explicit presentation is as follows:
\begin{align}
w_i^r&=1 \quad \qquad\qquad\qquad\qquad\mbox{for all $1\leq i\leq n$;}\label{athid-rel-def-Heckealgasscui1}\\[0.1em]
w_iw_j&=w_jw_i \qquad\qquad\qquad\quad\hspace{1.4mm}  \mbox{for all $1\leq i, j\leq n$;}\label{athid-rel-def-Heckealgasscui2}\\[0.1em]
h_{s_{i}}h_{s_{j}}&=h_{s_{j}}h_{s_{i}}\quad \quad\qquad\qquad\hspace{3.9mm}  \text{if $|i-j |> 1$ and $2\leq i, j\leq n-1$};\label{athid-rel-def-Heckealgasscui5}\\[0.1em]
h_{s_{i}}h_{s_{i+1}}h_{s_{i}}&=h_{s_{i+1}}h_{s_{i}}h_{s_{i+1}}\quad \quad\qquad \hspace{0.55mm}  \text{if $2\leq i\leq n-2$ and $n\geq 3$};\label{athid-rel-def-Heckealgasscui6}\\[0.1em]
h_{s_{0}}h_{s_{k}}&=h_{s_{k}}h_{s_{0}}\quad \quad\qquad\qquad\hspace{3.7mm}  \text{for all $2\leq k\leq n-2$};\label{athidaddcui-def-Heckealgaddcuiass6}\\[0.1em]
h_{s_{0}}h_{s_{n-1}}h_{s_{0}}&=h_{s_{n-1}}h_{s_{0}}h_{s_{n-1}}\quad \quad\qquad \hspace{0.55mm}\label{athidaddcui-def-Heckealgaddcui7}\\[0.1em]
h_{s_{i}}^{2}&=1+(q-q^{-1})h_{s_{i}} \quad \hspace{7mm}\mbox{for all $i=0$ or $2\leq i\leq n-1$;}\label{athid-rel-def-Heckealgasscuiiuc7}\\[0.1em]
h_{s_{i}}w_i=w_{i+1}h_{s_{i}}&-\Delta^{-2}\sum_{c_{1}< c_{2}}(\zeta^{c_2}-\zeta^{c_{1}})(q-q^{-1})F_{c_1}(w_{i})F_{c_2}(w_{i+1}) \hspace{2.5mm}\mbox{for $i=0$ or $2\leq i\leq n-1$;}\label{athid-rel-def-Heckealgassiuc8}\\[0.1em]
h_{s_{i}}w_{i+1}=w_{i}h_{s_{i}}&+\Delta^{-2}\sum_{c_{1}< c_{2}}(\zeta^{c_2}-\zeta^{c_{1}})(q-q^{-1})F_{c_1}(w_{i})F_{c_2}(w_{i+1}) \hspace{2.5mm}\mbox{for $i=0$ or $2\leq i\leq n-1$;}\label{athid-rel-def-Heckealgassiuc9}\\[0.1em]
h_{s_{i}}w_{l}&=w_{l}h_{s_{i}} \quad \qquad\qquad\qquad\hspace{2mm}\mbox{for all $l\not\equiv i, i+1$ (mod $n$) and $i=0$ or $2\leq i\leq n-1$,}\label{athid-rel-def-Heckealgassiuc10}
\end{align}
where $w_{0} :=w_{n}$ and in the expressions above, the sum is taken over all $1\leq c_{1}, c_{2}\leq r$ such that $c_{1}< c_{2}.$
\end{definition}

By using Proposition \ref{er-isomorphism7} again, we can easily get the following result.
\begin{proposition}\label{ERer-iucaddisomorphism8}
There is a canonical isomorphism between $\mathcal{H}_{r,n}^{2}$ and $\mathcal{C}_{r,n}^{2}$, which is explicitly described by $\phi_{2} :\mathcal{H}_{r,n}^{2}\rightarrow \mathcal{C}_{r,n}^{2}$ and $\psi_{2} :\mathcal{C}_{r,n}^{2}\rightarrow \mathcal{H}_{r,n}^{2}$, where $\phi_{2}$ is defined as follows$:$

\noindent for $1\leq j\leq n$,
\[\phi_{2}(t_{j})=w_{j},\]
\noindent for $i=0$ or $2\leq i\leq n-1$,
\[\phi_{2}(T_{s_{i}})=h_{s_{i}}-\Delta^{-2}(q-q^{-1})\sum_{c_{1}< c_{2}}F_{c_1}(w_{i})F_{c_2}(w_{i+1}),\]
\noindent and its inverse $\psi_{2}$ is defined as follows$:$

\noindent for $1\leq j\leq n$,
\[\psi_{2}(w_{j})=t_{j},\]
\noindent for $i=0$ or $2\leq i\leq n-1$,
\[\psi_{2}(h_{s_{i}})=T_{s_{i}}+\Delta^{-2}(q-q^{-1})\sum_{c_{1}< c_{2}}F_{c_1}(t_{i})F_{c_2}(t_{i+1}).\]
\end{proposition}

{\it Proof of Theorem \ref{iwahori-matsumoto-pre-cuimaintheorem}} Combining Propositions \ref{er-isomorphism7}, \ref{ERer-isomorphism8} and \ref{ERer-iucaddisomorphism8}, we can see that $\phi$ preserves the relations involving the generators $t_{1},\ldots, t_{n}$ and $T_{s_{0}},T_{s_{1}}\ldots, T_{s_{n-1}}$ in Definition \ref{Definition 2-4}, and $\psi$ preserves the relations involving the generators $w_{1},\ldots, w_{n}$ and $h_{s_{0}},h_{s_{1}}\ldots, h_{s_{n-1}}$ in Definition \ref{Definition 3-1}, respectively.

By definition, it is obvious that $\phi$ preserves the relations involving the generators $t_{1},\ldots, t_{n}$ and $T_{\rho}^{\pm 1}$ in Definition \ref{Definition 2-4}, and $\psi$ preserves the relations involving the generators $w_{1},\ldots, w_{n}$ and $h_{\rho}^{\pm 1}$ in Definition \ref{Definition 3-1}, respectively. Finally, it suffices to verify that $\phi$ preserves the relation (\ref{rel-def-Heckealg5}) and $\psi$ preserves the relation (\ref{athid-rel-def-Heckealg4}), which follows easily from their definitions. $\hfill{} \Box$\\

Recall that $\widehat{W}$ is the extended affine Weyl group of type $A$ with generators $\rho$ and $s_{i}$ ($0\leq i\leq n-1$). For each $\widehat{w}\in \widehat{W}$, let $\widehat{w}=$ $\rho^{k}s_{i_1}\cdots s_{i_{r}}$ be a reduced expression of $\widehat{w}.$ From the relations \eqref{athid-rel-def-Heckealg4}-\eqref{athid-rel-def-Heckealg7}, we get that $h_{\widehat{w}} :=h_{\rho}^{k}h_{s_{i_1}}\cdots h_{s_{i_{r}}}$ is independent of the choice of the reduced expression of $\widehat{w},$ that is, it is well-defined.

By Theorem \ref{iwahori-matsumoto-pre-cuimaintheorem} and Proposition \ref{coxeter-lemma-iso-pbwnbasis-4}, we can easily get the following result.
\begin{proposition}\label{cui-lemma-iso-pbwnbasisicu}
$\widehat{\mathcal{C}}_{r,n}^{\mathrm{aff}}$ has an $R$-basis consisting of the following elements$:$
\begin{align}\label{cui-lemma-iso-pbwnbasis-8888}
\big\{t_{1}^{\alpha_{1}}\cdots t_{n}^{\alpha_{n}}h_{\widehat{w}}\:|\:0\leq\alpha_{1},\ldots,\alpha_{n}\leq r-1,~\widehat{w}\in \widehat{W}\big\}.
\end{align}
\end{proposition}
\begin{proof}
From the relations \eqref{athid-rel-def-Heckealg3} and \eqref{athid-rel-def-Heckealg8}-\eqref{athid-rel-def-Heckealg10}, we can see that $\widehat{\mathcal{C}}_{r,n}^{\mathrm{aff}}$ is generated over $R$ by the elements of the form $t_{1}^{\alpha_{1}}\cdots t_{n}^{\alpha_{n}}h_{\widehat{w}}$ with $0\leq\alpha_{i}\leq r-1$ and $\widehat{w}\in \widehat{W}$. Thus, it suffices to prove that these elements are linearly independent over $R.$ By \eqref{athid-rel-def-Heckealg3} and \eqref{athid-rel-def-Heckealg8}-\eqref{athid-rel-def-Heckealg10} again and the morphism $\psi$ defined in Theorem \ref{iwahori-matsumoto-pre-cuimaintheorem}, we can get that
\begin{align}\label{bruhatorder-relations}
t_{1}^{\alpha_{1}}\cdots t_{n}^{\alpha_{n}}h_{\widehat{w}}=t_{1}^{\alpha_{1}}\cdots t_{n}^{\alpha_{n}}T_{\widehat{w}}+\sum_{\substack{\widehat{y}\prec\widehat{w}\\0\leq\beta_{i}\leq r-1}}t_{1}^{\beta_{1}}\cdots t_{n}^{\beta_{n}}T_{\widehat{y}},
\end{align}
where $\prec$ is the Bruhat order on $\widehat{W}.$ By Proposition \ref{coxeter-lemma-iso-pbwnbasis-4}, the set $\{t_{1}^{\alpha_{1}}\cdots t_{n}^{\alpha_{n}}T_{\widehat{w}}\:|\:0\leq\alpha_{i}\leq r-1\text{ and }\widehat{w}\in \widehat{W}\}$ is linearly independent. By \eqref{bruhatorder-relations}, we then get the desired result.
\end{proof}

Let $\widehat{\mathcal{H}}_{n}^{\mathrm{aff}}$ be the $R$-subalgebra of $\widehat{\mathcal{C}}_{r,n}^{\mathrm{aff}}$ generated by the elements $h_{s_{0}},\ldots,h_{s_{n-1}},$ and $h_{\rho}^{\pm 1},$ which is canonically isomorphic to the extended affine Hecke algebra of type $A,$ and we shall not distinguish between them. Thus, from Theorems \ref{iwahori-matsumoto-pre} and \ref{iwahori-matsumoto-pre-cuimaintheorem} we immediately get the following result.
\begin{corollary}
The extended affine Hecke algebra $\widehat{\mathcal{H}}_{n}^{\mathrm{aff}}$ of type $A$ is a subalgebra of the affine Yokonuma-Hecke algebra $\widehat{Y}_{r,n}.$
\end{corollary}

\noindent{\bf Acknowledgements.}
The author was partially supported by the National Natural Science Foundation of China (No. 11601273).



\vskip0.4cm

School of Mathematics, Shandong University, Jinan, Shandong 250100, P.R. China.

\emph{E-mail address}: cwdeng@amss.ac.cn

\end{document}